\documentclass[a4paper]{amsart}
\usepackage{geometry}
\usepackage{nicefrac}
\usepackage{caption}
\usepackage{kpfonts}
\usepackage{tikz}
\usetikzlibrary{cd}
\usetikzlibrary{decorations.pathreplacing}
\usepackage{ytableau}
\usepackage[enableskew]{youngtab}
\usepackage[utf8]{inputenc}
\usepackage{graphicx}
\usepackage{multirow}
\usepackage{mathrsfs}
\usepackage[title]{appendix}
\usepackage{xcolor}
\usepackage{textcomp}
\usepackage{manyfoot}
\usepackage{booktabs}
\usepackage{algorithm}
\usepackage{algorithmicx}
\usepackage{algpseudocode}
\usepackage{listings}
\usepackage[T1]{polski}
\usepackage[english]{babel}

\newtheorem{theorem}{Theorem}
\newtheorem{proposition}{Proposition}
\usepackage{enumerate}
\usepackage{hyperref}
\usepackage{cleveref}

\newtheorem{example}{Example}
\newtheorem{remark}{Remark}
\newtheorem{definition}{Definition}

\newtheorem{lemma}{Lemma}
\newtheorem{corollary}{Corollary}
\newtheorem*{theorem*}{Theorem}
\newtheorem*{proposition*}{Proposition}

\def\R{{\mathbb{R}}}
\def\N{{\mathbb{N}}}
\def\C{{\mathbb{C}}}
\def\Q{{\mathbb{Q}}}
\def\Z{{\mathbb{Z}}}

\DeclareMathOperator{\sign}{\mathrm{sign}}

\DeclareMathOperator{\GL}{\mathrm{GL}}

\DeclareMathOperator{\Aut}{\mathrm{Aut}}

\DeclareMathOperator{\trace}{\mathrm{trace}}
\DeclareMathOperator{\id}{\mathrm{id}}

\newcommand{\abs}[1]{\left| #1\right|}

\begin{document}

\title{Asymptotics of Plethysm}

\author{Tim Kuppel}
\address{Mathematisches Institut, Universit\"at Bonn, Endenicher Allee 60, Bonn, Germany}
\email{tkuppel@math.uni-bonn.de}


\begin{abstract}We study multiplicities $a^{d\lambda}_{\mu,(dk)}$ of highest weight representations $\mathbb S_{d\lambda}(\C^n)$, $\lambda\vdash pk$, of length at most $p$, in $\mathbb{S}_{\mu}(S^{dk}(\C^n))$, $\mu\vdash p$, so called plethysm coefficients, as $d$ tends to $\infty$.
	These are given by quasi-polynomials, which in the case of $S^p(S^{dk}(\C^n))$ can explicitly be computed by Pieri's rule.
	We show that for all but a finite, explicit list of $\lambda$'s the leading term is in fact constant and that 
	\begin{displaymath}
		a^{d\lambda}_{\mu,(dk)}\sim \frac{\dim V_\mu}{p!}c^{d\lambda}_{p,dk}
	\end{displaymath}
	as $d\to\infty$.
	In particular, we answer a conjecture of Kahle and Micha\l ek, going back to Howe.
\end{abstract}

\maketitle


\section{Introduction}\label{sec1}
The operation of \textit{Plethysm} was introduced within the context of symmetric functions by D. E. Littlewood in \cite{Lit36}.

Littlewood's motivation for introducing plethysm was classical invariant theory, namely determining the number of linearly independent homogeneous polynomials of fixed degree $d$ in the coefficients of polynomials in $n$ variables \cite[p. 305]{Lit44},
called \textit{covariants} of degree $d$ and order $n$ \cite[p. 31]{RK84}.
One well known example of such an invariant from high school is the discriminant $\delta=b^2-4ac$ of the polynomial $f(x,y)=ax^2+bxy+cy^2$, which under coordinate change just gets scaled by the determinant of the corresponding base change. This is up to scaling the only invariant of degree 2 and order 2.

Apart from its classical roots, plethysm also has applications in other areas of mathematics, stemming from the connection between symmetric functions and representation theory of the symmetric and general linear group.

The intimate connection between plethysm and representations of general linear groups gives rise to many applications, from whom we are just naming two.

For example, plethysm is used in geometric complexity theory (see \cite{Lan17} for details, in particular \cite[8.8-10]{Lan17} for the use of plethysm and arising problems),
which tries to contribute to the famous P versus NP problem.

Also, many important varieties in algebraic geometry come with an action of a general linear group $\GL(V)$,
but live in ambient spaces like the symmetric power $S^d(V)$ or wedge product $\bigwedge^d(V)$;
for example the Grassmann variety of $d$-dimensional spaces in $V$ lives in $\bigwedge^d(V)$ (cf. \cite[ch. 9]{Fu97}), and the Veronese variety of $d$-th powers of linear forms lives in $S^d(V)$ (cf. \cite[11.3, 13.3]{FH04}).
Hence, studying polynomials on these ambient spaces comes down to understanding the spaces $S^p(S^d(V))$ and $S^p(\bigwedge^d(V))$ together with their $\GL(V)$ action, which is exactly what plethysm is concerned with.

But still, plethysm is poorly understood, and only a few plethysms can be explicitly decomposed.
As a consequence, Stanley in \cite[Problem 9]{Sta99b} asks for a combinatorial description of plethysm coefficients, but this seems out of reach at the moment.
In fact, even deciding whether certain plethysm coefficients are positive is NP-hard \cite[Thm. 3.5]{FI20}.

Now consider the plethysm $\mathbb{S}_\mu(S^{dk}(V))$ for $\mu\vdash p$ a partition, and let $\lambda\vdash pk$.
It is natural to ask with what multiplicity the irreducible $\mathrm{Gl}(V)$-representation $\mathbb{S}_{d\lambda}(V)$ appears; let us denote the multiplicity by $a_{\mu,(dk)}^{d\lambda}$, see \cref{def:a} for the general definition.
Schur-Weyl duality implies that as $S_p\times\mathrm{Gl}(V)$-representation
\begin{displaymath}
	(S^{dk}(V))^{\otimes p}\cong\bigoplus\limits_{\mu\vdash p}V_\mu\otimes\mathbb S_{\mu}(S^{dk}(V)),
\end{displaymath}
where $V_\mu$ is the irreducible representation of the symmetric group $S_p$ corresponding to $\mu\vdash p$.
Moreover, Pieri's rule lets one compute the multiplicity $c^{d\lambda}_{p,dk}$ of $\mathbb{S}_{d\lambda}(V)$ in $(S^{dk}(V))^{\otimes p}$ combinatorially.
Thus, assuming that the multiplicity of $\mathbb{S}_{d\lambda}(V)$ is asymptotically equally distributed over the $\mathbb{S}_{\mu}(S^{dk}(V))$, $\mu\vdash p$, one expects
\begin{displaymath}
	a^{d\lambda}_{\mu,(dk)}\sim \frac{\dim V_\mu}{p!}c^{d\lambda}_{p,dk},
\end{displaymath}
given that $\sum_{\mu\vdash p}\dim V_\mu=p!$.

This note is meant to give a precise formulation and proof of a strengthening of this intuition, which in a slightly modified form was conjectured by Kahle and Micha\l ek in \cite[Conj. 4.3]{KM16} for arbitrary $p$ and all $\lambda$,
and proposed to Kahle and Micha\l ek by Mich\`ele Vergne (private communication with the second author of \cite{KM16}).
In \cite[Lemma 4.1]{KM16} a proof for \glqq non exceptional\grqq $\lambda$  whose parts are all distinct is given.
\begin{theorem*}[Thm. \ref{connection:MainThm}]
	Let $p,k\in\N$, and $\lambda\vdash pk$ with $l(\lambda)\leq p$.
	Then,
	\begin{enumerate}[(i)]
	\item
		if $\lambda$ is of the form (\glqq exceptional\grqq)
		\begin{displaymath}
			(pk),(k^{p}),(a^{p-1}),(b,c^{p-1}),(b^{p-1},c),
		\end{displaymath}
		we either have
		\begin{displaymath}
			a_{(p),(2dk)}^{2d\lambda}=a_{(1^{p}),((2d+1)k)}^{(2d+1)\lambda}=1,\quad a_{(p),((2d+1)dk)}^{(2d+1)\lambda}=a_{(1^p),(2dk)}^{2d\lambda}=0,\quad
			a_{\mu,(dk)}^{d\lambda}=0
		\end{displaymath}
		for all $d\geq 0$ and $\mu\vdash p$, $\mu\neq (p),(1^p)$, or
		\begin{displaymath}
			a_{(p),(dk)}^{d\lambda}=1,\quad
			a_{\mu,(dk)}^{d\lambda}=0
		\end{displaymath}
		for all $d\geq 0$ and $\mu\vdash p$, $\mu\neq(p)$,
	\item
		if $d=4$ and $\lambda=(2k,2k)$, then
		\begin{displaymath}
		\begin{array}{ll}
			a_{(4),(d)}^{(2d^2)}=
			\left\lfloor\frac{2d}{3}\right\rfloor-\frac{d}{2}+
			\begin{cases}
			1&d\mathrm{\ even}\\
			\frac{1}{2}&d\mathrm{\ odd}
			\end{cases},\quad &
			a_{(1^4),(d)}^{(2d^2)}=\left\lfloor\frac{2d}{3}\right\rfloor-\frac{d}{2}+
			\begin{cases}
			0&d\mathrm{\ even}\\
			\frac{1}{2}&d\mathrm{\ odd}
			\end{cases},\\[\bigskipamount]
			a_{(2,2),(d)}^{(2d^2)}=d-\left\lfloor\frac{2d}{3}\right\rfloor,&
			a_{(3,1),(d)}^{(2d^2)}=a_{(2,1^2),(dk)}^{(2d^2)}=0,
		\end{array}
		\end{displaymath}
		and if $\lambda=(b^2,c^2)$ for $b>c$, then $a_{\mu,(dk)}^{d\lambda}=a_{\mu,(d(k-a))}^{((b-c)^2)}$,
	\item
		and else $a_{\mu,(dk)}^{d\lambda}$ is a quasi-polynomial in $d$ of the same (positive) degree as $c_{p,dk}^{d\lambda}$ with constant leading term equal to $\frac{\dim(V_\mu)}{p!}$ times the leading term of $c_{p,dk}^{d\lambda}$
		for every $\mu\vdash p$.
	\end{enumerate}
\end{theorem*}
In fact, the above intuition also informs our proof.
Let us give an outline.
Fix $p,k\in\N$, $\lambda\vdash pk$ with $l(\lambda)\leq p$, a vector space $V$ with $n:=\dim(V)\geq p\geq l(\lambda)$, a maximal unipotent subgroup $U\subset\GL(V)$ as well as a maximal torus $T=(\C^*)^n\subset\GL(V)$, and define
\begin{displaymath}
	T_\lambda:=\{t\in T:t^\lambda=t_1^{\lambda_1}\ldots t_n^{\lambda_n}=1\},\quad
	A_d:=(S^{dk}(V))^{\otimes p},\quad
	B_d:=(A_d^U)^{T_\lambda}
\end{displaymath}
for $d\geq0$.
Then, crucially using Schur-Weyl duality we show the following.
\begin{proposition*}[Prop.\ \ref{connection:prop1}]
	The algebra $\bigoplus_{d\geq0}B_d$
	is finitely generated,
	equipped with a graded action of $S_p$, i.e., we have a group homomorphism
	\begin{displaymath}
		\beta:S_p\to\Aut(\bigoplus_{d\geq 0}B_d)
	\end{displaymath}
	whose image consists of graded algebra homomorphism,
	so that we get representations $\beta_d:S_p\to\GL(B_d)$ for each $d\geq 0$.
	Furthermore,
	the multiplicity of the Specht module $V_\mu$ for some $\mu\vdash p$ in $B_d$ equals $a_{\mu,(dk)}^{d\lambda}$,
	i.e., the multiplicity of $\mathbb{S}_{d\lambda}(V)$ in $\mathbb{S}_\mu(S^{dk}(V))$, and $\dim(B_d)$ equals $c_{p,dk}^{d\lambda}$, i.e., the multiplicity of $\mathbb{S}_{d\lambda}(V)$ in $A_d$.
	Also, $B_d$ is the space of highest weight vectors of weight $d\lambda$ in $(S^{dk}(V))^{\otimes p}$.
\end{proposition*}
Moreover, we have the following general result, which is a slight adaptation of \cite{How89} and brings the action of $S_p$ to the forefront.
\begin{theorem*}[Thm. \ref{connection:Howe}]
	Let $\beta:S_p\to\Aut(\bigoplus_{d\geq0}B_d)$ be the group homomorphism giving rise to representations $\beta_d:S_p\to\GL(B_d)$ for each $d\geq 0$ as in \cref{connection:prop1}, and define
	\begin{displaymath}
		PK:=\{\sigma\in S_p:\forall d\geq0\ \exists c\in\C^*:\beta_d(\sigma)=c\cdot \id\}.
	\end{displaymath}
	Then, if $PK=\{1\}$, we have
	\begin{displaymath}
		\lim\limits_{d\to\infty}\frac{f_\mu(d)}{\dim(B_d)}=\frac{\dim(V_\mu)}{p!}
	\end{displaymath}
	 for any $\mu\vdash p$, where $f_\mu(d)$ is defined as the multiplicity of $V_\mu$ in $B_d$.
\end{theorem*}
Thus, in order to proof \cref{connection:MainThm} we have to show $PK=\{1\}$.
As $PK$ is normal and the only non-trivial normal subgroup of $S_p$ for $p\neq 4$ is the alternating group $A_p$, the problem reduces to constructing highest weight vectors which are neither symmetric nor skew-symmetric, and in case of $p=4$ to constructing highest weight vectors which are not invariant under the Klein four group $V\subset A_4
$.
This is then carried out in \cref{connection:prop3}.
\subsection*{Acknowledgement}
This is part of my bachelor thesis, written at the university of Konstanz under the supervision of Prof. Micha\l{}ek during the summer of 2021.
Given my advisor's incredible support, time commitment and interest in my studies, it is my greatest pleasure to thank him, although I have to apologize for writing this up far to late!

During my studies, I was supported by the Studienstiftung des deutschen Volkes.

\section{Recollections}
In this section, we recall facts about symmetric functions, Schur functions, the intimate connection between representation theory of $\mathrm{Gl}_n(\C)$ and plethysm, and give an overview about known results concerning the asymptotic behaviour of plethysm coefficients.
In particular, no claim of originality is made, and readers with experience in these fields can safely skip ahead to section 3.
\subsection{Symmetric functions, representation theory of $\mathrm{Gl}_n(\C)$ and plethysm}
Let
\begin{displaymath}
	\Lambda_\Q=\varprojlim\limits_{n\to\infty}\Q[x_1,\ldots,x_n]^{S_n}
\end{displaymath}
be the ring of symmetric functions over $\Q$.
For $r\in\N_0$ we denote by 
\begin{displaymath}
	e_r,\
	h_r,\
	p_r
\end{displaymath}
the \textit{r-th elementary symmetric polynomial,\ r-th complete symmetric function} and \textit{r-th power sum},
as well as for $\lambda=(\lambda_1\geq \ldots\geq\lambda_l)$ a partition
\begin{displaymath}
	e_\lambda=\prod\limits_{i=1}^le_{\lambda_i},\
	h_\lambda=\prod\limits_{i=1}^lh_{\lambda_i},\
	p_\lambda=\prod\limits_{i=1}^lp_{\lambda_i}.
\end{displaymath}
Then, in fact, the $(e_{\lambda})_{\lambda}$ and $(h_{\lambda})_{\lambda}$, indexed over all partitions, form a $\mathbb Q$-basis of $\Lambda_{\mathbb Q}$, and the families $(e_r)_{r\in\N_0}$ and $(h_r)_{r\in\N_0}$ as well as $(p_r)_{r\in\N_0}$ are algebraically independent families generating $\Lambda_{\Q}$, cf. \cite[Thm. 7.4.4, Cor. 7.5.2, Cor. 7.7.2]{Sta99}.
Thus, we can consider the following, expressing a kind of duality between elementary symmetric and complete symmetric functions.
\begin{definition}
	Let $\omega:\Lambda_\Q\to\Lambda_\Q$ be given by requiring
	\begin{displaymath}
		\omega(e_r)=h_r
	\end{displaymath}
	for all $r\geq0$, inducing a graded ring homomorphism $\omega:\Lambda_\Q\to\Lambda_\Q$.
\end{definition}
	In the following, we abbreviate semistandard Young tableau by SSYT.	
	We will in particular consider \textit{Schur functions} 
	\begin{displaymath}
		s_\lambda\coloneqq
		\sum\limits_{T\mathrm{\ SSYT\ of\ shape\ }\lambda}x^T
	\end{displaymath}
	indexed by partitions $\lambda$,
	where we use the convention $s_\emptyset=1$.
	These also form a $\Q$-basis of $\Lambda_\Q$, cf. \cite[Cor. 7.10.6]{Sta99}.
\begin{proposition}[{\cite[I.3, ex. 1]{Mac98}}]\label{symf:cor3}\label{symf:prop3}
	Let $\lambda=(\lambda_1,\ldots,\lambda_l)$ be a partition with dual $\mu=\lambda^T$.
	Then $\omega(s_\lambda)=s_{\mu}$.
	In particular, $\omega$ is an involution, i.e., $\omega^2$ is the identity map.
\end{proposition}
In representation theory, Schur functions appear as characters of irreducible highest weight representations.
\begin{theorem}[{\cite[Thm. A2.4]{Sta99}}]\label{symf:thm2}
	Let $V$ be a finite dimensional complex vector space,
	and $\lambda\vdash d$ a partition with $l(\lambda)\leq \dim(V)=n$.
	Then the character of $\mathbb{S}_\lambda(V)$ is $s_\lambda(x_1,\ldots,x_n)$, i.e.,
	\begin{displaymath}
		\chi_{\mathbb{S}_\lambda(V)}(M)=s_\lambda(m_1,\ldots,m_n)
	\end{displaymath}
	for $M\in\GL(V)$, where $m_1,\ldots,m_n$ denote the zeroes of the characteristic polynomial $\det(M-t\id)$ of $M$.
\end{theorem}
On the representation theoretic side, the involution $\omega$ correspond to the following.
\begin{definition}
	Let $W$ be a polynomial representation of $\GL(V)$,
	and write
	\begin{displaymath}
		W=\bigoplus\limits_{\lambda}\mathbb{S}_\lambda(V)^{\oplus a_\lambda}.
	\end{displaymath}
	 Then, we define
	 \begin{displaymath}
	 	W^T:=\bigoplus\limits_{\lambda}\mathbb{S}_{\lambda^T}(V)^{\oplus a_\lambda},
	 \end{displaymath}
	 i.e., $W^T$ is obtained from $W$ by replacing each irreducible component corresponding to a partition $\lambda$ by the irreducible component corresponding to $\lambda^T$.
\end{definition}
\subsection{What is plethysm?}
\begin{definition}
	Let $g\in\Lambda_\Q$.
	Then, as the power sum symmetric functions $p_1,p_2,\ldots$ generate $\Lambda_\Q$ and are algebraically independent, we get a unique $\Q$-algebra homomorphism $\Lambda_\Q\to\Lambda_\Q,f\mapsto f[g]$ by requiring
	\begin{displaymath}
		p_n[g]:=g(x_1^n,x_2^n,\ldots)
	\end{displaymath}
	for $n\in\N$.
	We call $f[g]$ the \textit{plethysm} or \textit{composition} of $f$ and $g$.
\end{definition}
The following shows that plethysm can also be understood in the context of representations of the general linear group.
\begin{proposition}[{\cite[p. 448]{Sta99}}]\label{plethysm:prop1}
	Let $\lambda,\mu$ be partitions, $V$ a complex finite dimensional vector space, $n=\dim V$.
	Then, the character of the $\GL(V)$-representation $\mathbb{S}_\mu(\mathbb{S}_\lambda(V))$ is 
	\begin{displaymath}
	s_\mu[s_\lambda](x_1,\ldots, x_{n}).
	\end{displaymath}
\end{proposition}
\begin{definition}\label{def:a}
	Let $\lambda,\mu,\pi$ be partitions with $\abs{\pi}=\abs{\lambda}\cdot\abs{\mu}$.
	We then define the \textit{plethysm coefficient} $a_{\mu\lambda}^\pi$ as the coefficient of $s_\pi$ in the plethysm $s_\mu[s_\lambda]$,
	or, by the preceding \cref{plethysm:prop1}, as the multiplicity of $\mathbb{S}_\pi(V)$ in $\mathbb{S}_\mu(\mathbb{S}_\lambda(V))$,
	where $V$ is a finite dimensional complex vector space of dimension at least $l(\pi)$.
\end{definition}
\begin{example}[{\cite[I.8, Ex. 9]{Mac98}}]\label{plethysm:ex1}
	It holds
	\begin{displaymath}
		h_2[h_n]=\sum\limits_{k=0}^{\lfloor\frac{n}{2}\rfloor}s_{(2n-2k,2k)}.
	\end{displaymath}
	We can also interpret this as the decomposition into irreducible $\GL(V)$ representations
	\begin{displaymath}
		S^2(S^n V)=\bigoplus_{\lambda}\mathbb{S}_\lambda(V),
	\end{displaymath}
	where the sum ranges of all partitions of $2n$ into 2 even parts, and $V$ is a complex finite dimensional vector space.
	Since furthermore
	\begin{displaymath}
		(S^n(V))^{\otimes 2}=S^2(S^n(V))\bigoplus\bigwedge^2(S^n(V)),\quad
		(S^n(V))^{\otimes 2}=\bigoplus_{k=0}^n\mathbb{S}_{(2n-k,k)}(V)
	\end{displaymath}
	by Schur-Weyl duality and Pieri's rule respectively,
	we also get
	\begin{displaymath}
		\bigwedge^2(S^n(V))=\bigoplus_{\lambda}\mathbb{S}_\lambda(V),
	\end{displaymath}
	where the sum ranges over all partitions of $\lambda$ into two odd parts.
\end{example}
The connection to representation theory also shows that the plethysm coefficients $a_{\mu,\lambda}^\pi$ for $\lambda,\mu,\pi$ partitions with $\abs{\pi}=\abs{\mu}\cdot\abs{\lambda}$ are non-negative, of which no combinatorial proof is known \cite[p. 499]{Sta99}.
\begin{proposition}[{\cite[I.8, Ex. 1]{Mac98}}]\label{plethysm:prop2}
	Let $f\in\Lambda^m_{\Q}$, $g\in\Lambda^n_{\Q}$. Then,
	\begin{displaymath}
		\omega(f[g])=
		\begin{cases}
			f[\omega(g)]&,\mathrm{\ }n\mathrm{\ even}\\
			\omega(f)[\omega(g)]&,\mathrm{\ }n\mathrm{\ odd}.
		\end{cases}
	\end{displaymath}
\end{proposition}
Interpreting this result representation theoretic, we get the following corollary,
as the operation $(-)^T$ corresponds to applying $\omega$ to the character by \cref{symf:cor3}.
\begin{corollary}\label{plethysm:cor2}
	For partitions $\lambda,\mu$ and a complex finite dimensional vector space $V$ of dimension at least $\abs{\lambda}\cdot\abs{\mu}$, we have
	\begin{displaymath}
		\left(\mathbb{S}_\lambda(\mathbb{S}_\mu(V))\right)^T=
		\mathbb{S}_{\lambda^T}(\mathbb{S}_{\mu^T}(V)),
	\end{displaymath}
	if $\abs{\mu}$ is odd, and
	\begin{displaymath}
		\left(\mathbb{S}_\lambda(\mathbb{S}_\mu(V))\right)^T=
		\mathbb{S}_{\lambda}(\mathbb{S}_{\mu^T}(V)),
	\end{displaymath}
	if $\abs{\mu}$ is even.
\end{corollary}
\begin{example}
	From \cref{plethysm:ex1} we know
	\begin{displaymath}
		S^2(S^2(V))=S^4(V)\oplus S_{(2,2)}(V),
	\end{displaymath}
	so
	\begin{displaymath}
		S^2(\bigwedge^2(V))=\bigwedge^4(V)\oplus S_{(2,2)}(V).
	\end{displaymath}
\end{example}
\subsection{What is known about asymptotics of plethysm}
As mentioned in the introduction, in \cite{FI20} it is shown that deciding whether plethysm coefficients are positive in general is NP-hard.
But Weintraub in \cite[Conj. 2.11]{Wei90} conjectured the following.
\begin{theorem}[{\cite{BCI11},\cite{MM14}}]\label{plethysm:Weintraub}
	Let $\lambda\vdash pk$ with $k$ even and all parts of $\lambda$ even.
Then, $a_{(p),(k)}^\lambda\geq 1$, or equivalently the multiplicity of $\mathbb{S}_\lambda(V)$ in $S^p(S^k(V))$ is positive.
\end{theorem}
Weintraub observed that certain similar sequences of plethysm coefficients stabilize \cite{Wei90},
which Brion generalized in \cite{Bri93}.
Note that if we specialize $\nu=\tilde\nu=\emptyset$ and $\lambda=(1)$,
then this agrees with our \cref{connection:MainThm}.
\begin{theorem}[Sec. 2.6, Cor. 1 in \cite{Bri93}]
	Let $\mu\vdash p$, and $\nu,\tilde\nu,\lambda$ be partitions with $\abs{\nu}=\abs{\tilde\nu}$,
	where $\lambda=\left(l_1^{a_1},\ldots,l_q^{a_q}\right)$ for some $l_1>\ldots>l_q$ and positive integers $a_1,\ldots,a_q$.
	Then,
	\begin{displaymath}
		a_{\mu,\tilde\nu+d\lambda}^{\nu+pd\lambda}
	\end{displaymath}
	is an increasing sequence in $d$ which stabilizes for $d$ so that
	\begin{displaymath}
		\tilde\nu_i-\tilde\nu_{i+1}+d(\lambda_i-\lambda_{i+1})\geq
		p(\tilde\nu_1+\ldots+\tilde\nu_i)-\nu_1-\ldots-\nu_i
	\end{displaymath}
	for all $i\in\{a_1,a_1+a_2,\ldots,a_1+\ldots+a_q\}$.
\end{theorem}
In order to deduce asymptotic behaviour of the multiplicity of $\mathbb{S}_{d\lambda}(V)$ in $(S^{dk}(V))^{\otimes p}$, where $\lambda\vdash pk$ with $l(\lambda)\leq p$,
we are going to make use of lattice point counting, in particular Ehrhart theory, following \cite{KM16},
where Ehrhart theory is used to study the asymptotics of the plethysm coefficients we are interested in.
	
	Recall that a \textit{quasi-polynomial of degree} $n$ is a function $q:\N\to\N$ of the form
$q(d)=d^nc_n(d)+\ldots+c_0(d)$, where the $c_0,\ldots,c_n:\N\to\Z$ are periodic functions, i.e., 
$c_i(d+p)=c_i(d)$ for some $p\in\N$ and all $d\in\N$, and $c_n$ is not constant 0.
We call the minimal such $p$ its \textit{period}, and $c_n$ the leading term.

	For a polytope $P\subset\R^N$ we have its \textit{lattice-point enumerator}
	\begin{displaymath}
		\mathcal{L}_P(d):=\#(dP\cap\Z^N).
	\end{displaymath}
	Famously, Ehrhart showed that for a rational polytope $P\subset\R^N$, $\mathcal{L}_P(d)$ is a quasi-polynomial in $d$ of degree $\dim(P)$,
	whose period divides the least common multiple of the denominators of the coordinates of the vertices of $P$.
	In particular, if $P$ is a lattice polytope then $\mathcal{L}_P(d)$ is a polynomial.
	
We now define a polytope encoding Pieri's rule, as in \cite[Def. 3.3, Prop. 3.4]{KM16}.
\begin{definition}
	Let $p,k\in\N$, $\lambda\vdash pk$ with $l(\lambda)\leq p$.
	Furthermore, denote coordinates on $\R^2\times\R^3\times\ldots\times\R^{p}$ by $(x_1^1,x_2^1,x_1^2,x_2^2,\ldots,x_1^{p-1},\ldots,x_{p}^{p-1})$, set $x_1^0=k, x_2^0=\ldots=x_{p}^0=0$,
	and define the rational polytope $P_{k,p}^{\lambda}$ by the constraints
	\begin{enumerate}[(i)]
		\item
			$x_i^j\geq 0$ for all $1\leq i\leq p$ and $1\leq i\leq j+1$,
		\item
			$\sum\limits_{l=1}^jx_{i+1}^l\leq\sum\limits_{l=1}^{j-1}x_{i}^l$ for all $1\leq i\leq j\leq p-1$,
		\item
			$\sum\limits_{i=1}^px^j_i=k$ for all $1\leq i\leq p$, and
		\item
			$\sum\limits_{0\leq j}x_i^j=\lambda_i$ for all $1\leq i\leq p$.
	\end{enumerate}
	Constraints (i) and (iii) imply that $P_{k,p}^\lambda$ is bounded.
	
	Furthermore,
	a bounded set given by linear inequalities with integer coefficients is a rational polytope by \cite[48]{BR15}, so $P_{k,p}^\lambda$ is in fact a rational polytope.
\end{definition}
\begin{proposition}\label{tensorasymp:prop1a}
	The number of lattice points in $P_{k,p}^\lambda$, i.e., points in $P_{k,p}^\lambda\cap\Z^{2+3+\ldots+p}$, is the number of SSYTs of shape $\lambda$ filled with $k$ 1's, $\ldots$, $p$'s.
\end{proposition}
\begin{proof}
	We interpret $x_i^j$ for $1\leq i\leq p-1$, $1\leq j\leq i+1$ as the number of boxes we add in the $j$-th step and $i$-th row to a Young tableau according to Pieri's rule.
	Constraint (i) assures that we do not subtract boxes, constraint (ii) assures that after each step we add at most one box in each column and still obtain a Young diagram,
	constraint (iii) assures that we add $k$ boxes in each step and (iv) assures that the SSYT we get is of shape $\lambda$.
\end{proof}
Kahle and Micha\l ek in \cite{KM16} showed the following, see in particular their Thm. 1.1.
\begin{proposition}\label{plethysm:QuasiPolynomial}
	Let $\lambda\vdash pk$, $\mu\vdash p$.
	Then, $a_{\mu,(dk)}^{d\lambda}$ is a quasi-polynomial in $d$.
\end{proposition}
\begin{example}[{\cite[Ex. 1.3]{KM16}}]
	Let $\lambda=(31,3,2,2,2)$ and define
	\begin{align*}
		p_1(d):=\frac{1}{720}d^3+\frac{1}{20}d^2-\frac{289}{720}d,\quad
		p_2(d):=\frac{1}{8}d+\frac{5}{8},\quad
		p_3(d):=-\frac{1}{6}d+\frac{1}{3},\quad
		p_4(d):=-\frac{1}{3}d+\frac{7}{12},\\
		A(d):=p_1+p_2\left\lfloor\frac{d}{2}\right\rfloor+p_3\left\lfloor\frac{d}{3}\right\rfloor+\left(p_4+\frac{1}{2}\left\lfloor\frac{d}{3}\right\rfloor\right)\left\lfloor\frac{1+d}{3}\right\rfloor+\frac{1}{4}\left(\left\lfloor\frac{1+d}{3}\right\rfloor^2+\left\lfloor\frac{d}{4}\right\rfloor-\left\lfloor\frac{3+d}{4}\right\rfloor\right).
	\end{align*}
	Then
	\begin{displaymath}
		a_{(5),(8d)}^{d\lambda}=
		A(d)+
		\begin{cases}
			1&d\equiv0\pmod{5}\\
			\frac{3}{5}&d\equiv1\pmod{5}\\
			\frac{4}{5}&d\equiv2,3,4\pmod{5}
		\end{cases}.
	\end{displaymath}
	Note that in this case $a_{(5),(8d)}^{d\lambda}$ is a quasi-polynomial whose leading term is constant.
	We shall see in \cref{connection:MainThm} that this is always the case.
\end{example}
\section{The argument}
In this section, we provide a proof of \cref{connection:MainThm}.
\subsection{Asymptotic behaviour of multiplicities in tensor products}
In this section more closely study the asymptotic behaviour of the multiplicity of $\mathbb{S}_{d\lambda}(V)$ in $(S^{dk}(V))^{\otimes p}$ using Pieri's rule, where $\mu\vdash p$ and $\lambda\vdash pk$ with $l(\lambda)\leq p$,
which we then relate to the plethysm coefficients we are interested in.

In order to do so, we first have to understand when the multiplicity is non-negative.
\begin{lemma}\label{tensorasymp:lemma1}
	Let $p,k\in\N$, and $\lambda\vdash pk$ with $l(\lambda)\leq p$.
	Then, there exists a SSYT of shape $\lambda$ filled with $k$ 1's, $k$ 2's, $\ldots$, $k$ $p$'s.
\end{lemma}
\begin{proof}
	We use induction on $p$.
	For $p=1$, $\lambda=(k)$ is the only partition of $pk$ of length at most $p$, and 
	filling the Young diagram of shape $(k)$ with $k$ 1's is a tableau with the required property.
	
	Now assume the claim is true for all $\mu\vdash pk$ with $l(\mu)\leq p$,
	and let $\lambda=(\lambda_1,\ldots,\lambda_{p+1})\vdash(p+1)k$ with $l(\lambda)\leq p+1$.
	
	Since $\lambda_1\geq\ldots\geq\lambda_{p+1}$,
	we have $\lambda_{p+1}\leq k$ and $\lambda_1\geq k$.
	Therefore, we may choose $i_0$ such that $\lambda_{i_0+1}<k\leq\lambda_{i_0}$.
	We then for each $i_0<j\leq l(\lambda)$ cross out the rightmost $\lambda_{j}-\lambda_{j+1}$ boxes in the $j$-th row,
	and the rightmost $k-\lambda_{i_0+1}$ boxes in the $i_0$-th row.
	By the choice of $i_0$
	we obtain a Young diagram of some shape $\lambda'$, where $\lambda'\vdash pk$.
	Since $l(\lambda)\leq p+1$ and $\lambda_{p+1}<k$, we  have $l(\lambda')\leq p$.
	By the induction hypothesis we find a SSYT of shape $\lambda'$ filled with $k$ 1's, $\ldots$, $p$'s.
	If we now add back all boxes we crossed out before and fill them with $p+1$, we obtain a SSYT with the required property, as we have crossed out at most one box in each column and only the rightmost boxes in each row where we crossed something out.
	\begin{figure}[H]
	\centering
	\begin{ytableau}
	 1&1&1&1&2&2\\
	 2&2&3&3&3\\
	 3&\varprod&\varprod&\varprod\\
	 \varprod
	\end{ytableau}
	\caption*{Example with $p=3$ and $k=4$.}
	\end{figure}
\end{proof}
In order to find out when multiple such tableaux exist, we will repeatedly make use of the following reduction.
\begin{lemma}\label{tensorasymp:lemma2}
	Let $p,k\in\N$ and $\lambda=(\lambda_1,\ldots)\vdash pk$ with $\lambda_1=p$, $\mu\vdash p$. Let $\lambda':=(\lambda_2,\ldots)$ and assume $n:=\dim(V)\geq l(\lambda)$. Then,
	$a_{\mu,(1^k)}^\lambda=a_{\mu,(1^{k-1})}^{\lambda'}$.
\end{lemma}
\begin{proof}
	See \cite[Lemma 3.2]{KM16}.
\end{proof}
Using this, we can proof the following corollaries.
\begin{corollary}\label{plethysm:cor1}
	Let $p,k\in\N$, $\mu\vdash p$ and $\lambda\vdash pk$ with $l(\lambda)=p$, and set $\lambda':=(\lambda_1-\lambda_p,\ldots,\lambda_{p-1}-\lambda_p)$.
	Then,
	\begin{displaymath}
		a_{\mu,(k)}^\lambda=a_{\mu,(k-\lambda_p)}^{\lambda'}
	\end{displaymath}
	if $\lambda_p$ is even,
	and
	\begin{displaymath}
		a_{\mu,(k)}^\lambda=a_{\mu^T,(k-\lambda_p)}^{\lambda'}
	\end{displaymath}
	if $\lambda_p$ is odd.
\end{corollary}
\begin{proof}
	Set $\tilde\lambda:=(\lambda_1-1,\ldots,\lambda_p-1)$.
	For odd $k$ we have by \cref{symf:cor3} and \cref{plethysm:prop2}
	\begin{displaymath}
		a_{\mu,(k)}^\lambda=
		a_{\mu^T,(1^k)}^{\lambda^T}\overset{\cref{tensorasymp:lemma2}}{=}
		a_{\mu^T,(1^{k-1})}^{(\tilde\lambda)^T}=
		a_{\mu^T,(k-1)}^{\tilde\lambda},
	\end{displaymath}
	and for even $p$
	\begin{displaymath}
		a_{\mu,(k)}^\lambda=
		a_{\mu,(1^k)}^{\lambda^T}=
		a_{\mu,(1^{k-1})}^{(\tilde\lambda)^T}=
		a_{\mu^T,(k-1)}^{\tilde\lambda}.
	\end{displaymath}
	Hence, we always have $a_{\mu,(k)}^\lambda=a_{\mu^T,(k-1)}^{\tilde\lambda}$.
	Doing this $\lambda_p$ times, the claim follows.
\end{proof}
\begin{definition}\label{def:c}
	Let $p,k\in\N$, $\lambda\vdash pk$.
	Then, we define $c_{p,k}^\lambda$ as the coefficient of $s_\lambda$ in $h_k^p$,
	or equivalently as the multiplicity of $\mathbb{S}_\lambda(V)$ in $(S^k(V))^{\otimes p}$, where $\dim(V)\geq l(\lambda)$, or the number of SSYTs of shape $\lambda$ filled with $k$ 1's, $\ldots$, $p$'s by Pieri's rule.
\end{definition}
\begin{corollary}\label{tensorasymp:cor1}
	Let $p,k\in\N$, and $\lambda=(\lambda_1,\ldots,\lambda_p)\vdash pk$ with $l(\lambda)=p$.
	Furthermore, assume $\dim(V)\geq p$.
	Then, $c_{p,k}^\lambda=c_{p,k-\lambda_p}^{\lambda'}$, where $\lambda':=(\lambda_1-\lambda_p,\ldots,\lambda_{p-1}-\lambda_p,0)$.
\end{corollary}
\begin{proof}
	By Schur-Weyl duality we have
	\begin{displaymath}
		(S^k(V))^{\otimes p}=
		\bigoplus\limits_{\mu\vdash p}V_\mu\otimes\mathbb{S}_\mu(S^k (V)),\quad
		(S^{k-\lambda_p}(V))^{\otimes p}=
		\bigoplus\limits_{\mu\vdash p}V_\mu\otimes\mathbb{S}_\mu(S^{k-\lambda_p} (V)).
	\end{displaymath}
	Therefore, the multiplicity of $\lambda$ in $(S^k (V))^{\otimes p}$ is
	\begin{displaymath}
		\sum\limits_{\mu\vdash p}\dim(V_\mu)\cdot a_{\mu,(k)}^\lambda,
	\end{displaymath}
	and the multiplicity of $\lambda'$ in $(S^{k-\lambda_p} (V))^{\otimes p}$ is
	\begin{displaymath}
		\sum\limits_{\mu\vdash p}\dim(V_\mu)\cdot a_{\mu,(k-\lambda_p)}^{\lambda'}.
	\end{displaymath}
	By the preceding corollary \cref{plethysm:cor1} we have $a_{\mu,(k)}^\lambda=a_{\mu,(k-\lambda_p)}^{\lambda'}$,
	if $\lambda_p$ is even,
	as well as $a_{\mu,(k)}^\lambda=a_{\mu^T,(k-\lambda_p)}^{\lambda'}$, if $\lambda_p$ is odd.
	Hence, for even $\lambda_p$ we see that the multiplicity of $\lambda$ in $(S^k (V))^{\otimes p}$
	equals the multiplicity of $\lambda'$ in $(S^{k-\lambda_p} V)^{\otimes p}$.
	But, by the Hook length formula \cite[4.12]{FH04} the dimensions of $V_\mu$ and $V_{\mu^T}$ are the same for any $\mu\vdash p$.
	Therefore, the claim is also true for odd $\lambda_p$.
\end{proof}
We are now ready to state when multiple tableaux as in \cref{tensorasymp:lemma1} exist.
\begin{proposition}\label{tensorasymp:prop1}
	Let $p,k\in\N$.
	Then, for partitions of $pk$ of the form 
	\begin{displaymath}
	(pk),(k^{p}),(a^{p-1}),(b,c^{p-1}),(b^{p-1},c)
	\end{displaymath}
	with integers $a, b>c$,
	there is exactly one SSYT of that shape filled with $k$ 1's, $\ldots$, $p$'s.
	Furthermore, for all other partitions $\lambda\vdash pk$ with $l(\lambda)\leq p$, there are at least two such SSYTs.
\end{proposition}
\begin{proof}
	For $\lambda=(k^p)$ the only SSYT filled with $k$ 1's, 2's,$\ldots$, $p$'s is the one filled with 1's in the first row, 2's in the second row, and so on.
	
	So let $\lambda\vdash pk$ with $l(\lambda)=p$ and $\lambda_p<k$.
	By Pieri's rule the number of SSYTs of shape $\lambda$ filled with $k$ 1's, $\ldots$, $p$'s is exactly the multiplicity of $\lambda$ in $(S^k (V))^{\otimes p}$, where we assume $\dim(V)\geq l(\lambda)=p$.
	Using the preceding \cref{tensorasymp:cor1}, this equals the multiplicity of $\lambda':=(\lambda_1-\lambda_p,\ldots,\lambda_{p-1}-\lambda_p,0)$ in $(S^{k-\lambda_p} (V))^{\otimes p}$,
	which again by Pieri's rule equals the number of SSYTs of shape $\lambda'$ filled with $k-\lambda_p$ 1's, 2's,$\ldots$,$p$'s.

	Furthermore, the partitions of $pk$ of the form $(a^{p-1},b), (a,b^{p-1})$ with $a>b$ are exactly those which after subtracting $b$ from each part are those of length at most $p-1$ for whom we claim that there is only one SSYT with the required property.
	Hence, it is enough to consider partitions of length at most $p-1$.
	
	First, we consider the partitions where we claim exactly one SSYT exists.
	For $\lambda=({kp})$ we only have one SSYT of shape $\lambda$ filled with $k$ 1's, $\ldots$, $p$'s,
	as entries along this single row have to be non decreasing.
	Therefore,
	only the case where $p\geq 2$ and $(a^{p-1})$ is a partition of $pk$ for some integer $a$ remains to be investigated.
	To this end, we use induction on $p\geq 2$.
	
	For $p=2$, we get the partition $(2k)$, and nothing has to be done.
	So assume that for some $p\geq 2$ and any $k$ such that there is an integer $a$ with $(a^{p-1})\vdash pk$ there is exactly one SSYT of shape $(a^{p-1})$ filled with $k$ 1's, $\ldots$, $p$'s.
	
	Let $k$ and $a$ be such that $(a^p)\vdash (p+1)k$.
	Note that we have $k<a$ and $2k>a$, since $p\geq 2$.
	As entries in a SSYT are non-decreasing along rows and increasing along columns,
	all $k$ $p+1$'s in a SSYT of shape $(a^p)$ filled with $k$ 1's, $\ldots$, $p+1$'s must be in the rightmost $k$ boxes of the last row.
	Therefore,
	the number of SSYTs of shape $(a^p)$ filled with $k$ 1's, $\ldots$, $p+1$'s equals the number of SSYTs of shape $(a^{p-1},a-k)\vdash pk$ filled with $k$ 1's, $\ldots$, $p$'s.
	By the above argument, which allows us to only consider partitions of length at most $p-1$,
	this equals the number of SSYTs of shape $(k^{p-1})\vdash p(2k-a)$ filled with $2k-a$ 1's, $\ldots$, $p$'s.
	The induction hypothesis yields that this number is 1, as required.
	
	Now let $\lambda\vdash pk$ with $2\leq l(\lambda)\leq p-1$.
	In particular, we have $\lambda_1>k$, justifying the following constructions.
	First, we assume that not all parts of $\lambda$ are equal,
	and construct to different SSYTs with the required property.
	
	Choose $i_0$ such that $\lambda_{i_0+1}<k\leq\lambda_{i_0}$.
	We then for each $i_0<j\leq l(\lambda)$ cross out the rightmost $\lambda_{j}-\lambda_{j+1}$ boxes in the $j$-th row,
	and the rightmost $k-\lambda_{i_0+1}$ boxes in the $i_0$-th row.
	By the choice of $i_0$
	we obtain a Young diagram of some shape $\lambda'$, where $\lambda'\vdash (p-1)k$.
	Since $l(\lambda)\leq p-1$, we in particular have $l(\lambda')\leq p-1$.
	Hence, by \cref{tensorasymp:lemma1} we find a SSYT of shape $\lambda'$ filled with $k$ 1's, $\ldots$, $p-1$'s.
	If we now add back all boxes we crossed out and fill them with $p$, we obtain a SSYT of shape $\lambda$ filled with $k$ 1's, $\ldots$, $p$'s.
	
	We obtain another SSYT with the required property in the following way.
	Choose $i_1$ such that $\lambda_1-\lambda_{i_1+1}\geq k>\lambda_1-\lambda_{i_1}$.
	We then for each $1\leq j\leq i_1$ cross out the $\lambda_j-\lambda_{j+1}$ rightmost boxes in the $j$-th row,
	and the rightmost $k-(\lambda_1-\lambda_{i_1})$ boxes in the $i_1$-th row.
	By the choice of $i_1$ we obtain a Young diagram of some shape $\lambda'$ where $\lambda'\vdash (p-1)k$.
	Then, we proceed as before.

	Since not all parts of $\lambda$ are equal and $\lambda_1>k$,
	we have crossed out $k$ boxes in two distinct ways.
	Therefore, we get two distinct SSYTs of shape $\lambda$ filled with $k$ 1's, $\ldots$, $p$'s.

	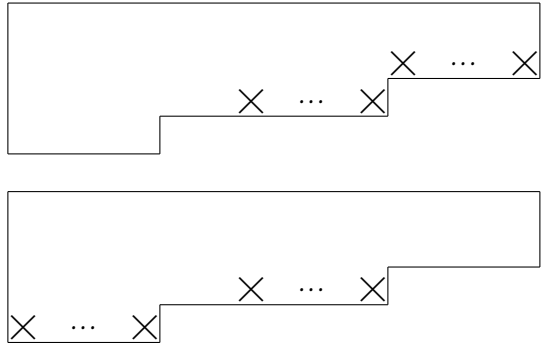
\begin{figure}[H]
		\begin{center}
		\begin{tikzpicture}
			\draw (0,0)--(2,0);
			\node at (.2,.2) {$\varprod$};
			\node at (1,.2) {$\ldots$};
			\node at (1.8,.2) {$\varprod$};
			\draw (2,0) -- (2,.5);
			\draw (2,.5)--(5,.5);
			\node at (3.2,.7) {$\varprod$};
			\node at (4,.7) {$\ldots$};
			\node at (4.8,.7) {$\varprod$};
			\draw (5,.5)--(5,1);
			\draw (5,1)--(7,1);
			\draw (7,1)--(7,2);
			\draw (7,2)--(0,2);
			\draw (0,2)--(0,0);
		\begin{scope}[shift={(-8,2.5)}]
			\draw (8,0)--(10,0);
			\node at (13.2,1.2) {$\varprod$};
			\node at (14,1.2) {$\ldots$};
			\node at (14.8,1.2) {$\varprod$};
			\draw (10,0) -- (10,.5);
			\draw (10,.5)--(13,.5);
			\node at (11.2,.7) {$\varprod$};
			\node at (12,.7) {$\ldots$};
			\node at (12.8,.7) {$\varprod$};
			\draw (13,.5)--(13,1);
			\draw (13,1)--(15,1);
			\draw (15,1)--(15,2);
			\draw (15,2)--(8,2);
			\draw (8,2)--(8,0);
		\end{scope}
		\end{tikzpicture}
		\caption*{schematic picture of both ways to cross out}
		\end{center}
	\end{figure}

	Lastly, let $\lambda\vdash pk$ with $2\leq l(\lambda)<p-1$ and $\lambda_1=\lambda_2=\ldots$.
	In particular, all parts of $\lambda$ are greater than $k$.
	We now cross out the rightmost $k$ boxes in the last row of the Young diagram of shape $\lambda$,
	and obtain a Young diagram of some shape $\lambda'$, where $\lambda'\vdash (p-1)k$.
	As there are more than $k$ boxes in the last row, not all parts of $\lambda'$ are equal.
	Furthermore, $l(\lambda')\leq p-2$, since $l(\lambda)\leq p-2$.
	Therefore, we find two distinct SSYTs of shape $\lambda'$ filled with $k$ 1's, $\ldots$, $p-1$'s by what was shown before.
	If we add back the $k$ boxes we crossed out before and fill them with $p$ in each of these two SSYTs,
	we obtain two distinct SSYTs with the required property, proving the proposition.
\end{proof}

\begin{corollary}\label{tensorasymp:cor2}
	Let $\lambda\vdash pk$ with $l(\lambda)\leq p$ for some $p,k\in\N$.
	Then,
	$c_{p,dk}^{d\lambda}$ is constantly 1 for $\lambda$ of the form
	\begin{displaymath}
		(pk),(1^{pk}),(a^{p-1}),(b,c^{p-1}),(b^{p-1},c),
	\end{displaymath}
	and otherwise a quasi-polynomial in $d$ of positive degree with constant leading term.
\end{corollary}
\begin{proof}
	By definition $c_{dk,p}^{d\lambda}$ is the number of SSYTs of shape $d\lambda$ filled with $dk$ 1's, $\ldots$, $p$'s, i.e.,
	\begin{displaymath}
		\#\left(P_{dk,p}^{d\lambda}\cap\Z^{2+3+\ldots+p}\right).
	\end{displaymath}
	By \cref{tensorasymp:prop1}, there is exactly one such SSYT for $\lambda$ of the form
	\begin{displaymath}
		(pk),(1^{pk}),(a^{p-1}),(b,c^{p-1}),(b^{p-1},c),
	\end{displaymath}
	and hence $c_{dk,p}^{d\lambda}$ is constantly 1 for those partitions.
	So assume $\lambda$ is not of this form.

	As $P_{dk,p}^{d\lambda}=dP_{k,p}^{\lambda}$ and $P_{k,p}^{\lambda}$ is a rational polytope,
	the number of SSYTs of shape $d\lambda$ filled with $dk$ 1's, $\ldots$, $p$'s is a quasi-polynomial in $d$,
	if $\dim(P_{k,p}^\lambda)>0$.
	As there are at least two such SSYTs of shape $\lambda$ by \cref{tensorasymp:prop1},
	the affine space spanned by $P_{k,p}^\lambda$ contains a line, and therefore $\dim(P_{k,p}^\lambda)>0$; as required.
	We shall see in \cref{connection:cor1} that the leading coefficient is constant.
\end{proof}
\subsection{Relating asymptotics of plethysm to multiplicities in tensor products}
Let $p,k\in\N$, and let $V$ be a finite dimensional complex vector space, $n=\dim(V)$.

Consider the commutative graded algebra $\bigoplus_{d\geq 0}(S^{dk}(V))^{\otimes p}$, which clearly is finitely generated and an integral domain.
As $(S^{dk}(V))^{\otimes p}$ for an any $d\geq0$ is a representation of both $\GL(V)$ and $S_p$, and their actions commute,
elements of both $\GL(V)$ and $S_p$ give graded algebra automorphisms of $\bigoplus_{d\geq0}(S^{dk}(V))^{\otimes p}$ which commute.
Here, $S_p$ acts in the right via $x_1\otimes\ldots x_p\cdot\sigma=x_{\sigma(1)}\otimes\ldots\otimes x_{\sigma(p)}$.

In general, the action of a torus $T$ or $\GL(V)$ on an algebra $A$ is called rational if every $a\in A$ is contained in a finite dimensional subspace $A'\subset A$ on which $T$ or $\GL(V)$ respectively acts rationally.
In particular, $\GL(V)$ acts rationally on $\bigoplus_{d\geq0}(S^{dk}(V))^{\otimes p}$,
as any element is contained in $\bigoplus_{N\geq d\geq0}(S^{dk}(V))^{\otimes p}$ for some $N\in\N$, which is a rational representation of $\GL(V)$.

From now on, fix $p,k\in\N$, $\lambda\vdash pk$ with $l(\lambda)\leq p$, a vector space $V$ with $n:=\dim(V)\geq p\geq l(\lambda)$, a Borel $B\subset\GL(V)$ with maximal unipotent subgroup $U\subset B$ and torus $T\subset B$, so $B=TU$.
We define
\begin{displaymath}
	T_\lambda:=\{t\in T:t^\lambda=t_1^{\lambda_1}\ldots t_n^{\lambda_n}=1\},\quad
	A_d:=(S^{dk}(V))^{\otimes p},\quad
	B_d:=(A_d^U)^{T_\lambda}
\end{displaymath}
for $d\geq0$.
\begin{proposition}\label{connection:prop1}
	The algebra $\bigoplus_{d\geq0}B_d$
	is finitely generated,
	equipped with an action of $S_p$ via graded algebra automorphisms, i.e., we have a group homomorphism
	\begin{displaymath}
		\beta:S_p\to\Aut(\bigoplus_{d\geq 0}B_d)
	\end{displaymath}
	whose image consists of graded algebra homomorphism,
	so that we get representations $\beta_d:S_p\to\GL(B_d)$ for each $d\geq 0$.
	Furthermore,
	the multiplicity of the Specht module $V_\mu$ for some $\mu\vdash p$ in $B_d$ equals $a_{\mu,(dk)}^{d\lambda}$,
	i.e., the multiplicity of $\mathbb{S}_{d\lambda}(V)$ in $\mathbb{S}_\mu(S^{dk}(V))$, and $\dim(B_d)$ equals $c_{p,dk}^{d\lambda}$, i.e., the multiplicity of $\mathbb{S}_{d\lambda}(V)$ in $A_d$.
	Also, $B_d$ is the space of highest weight vectors of weight $d\lambda$ in $(S^{dk}(V))^{\otimes p}$.
\end{proposition}
\begin{proof}
	The general linear group $\GL(V)$ acts on each graded piece of $\bigoplus_{d\geq0}A_d$, so
	\begin{displaymath}
		\left(\bigoplus_{d\geq0}A_d\right)^U=
		\bigoplus_{d\geq0}A_d^U.
	\end{displaymath}
	By {\cite[Thm. 9.4]{Gro97}} this is a finitely generated algebra,
	with graded pieces spanned by highest weight vectors \cite[p. 115]{Fu97},
	on whom the torus $T$ and therefore in particular $T_\lambda$ acts.
	So we furthermore get
	\begin{displaymath}
		\left(\bigoplus_{d\geq0}A_d^U\right)^{T_\lambda}=
		\bigoplus_{d\geq0}(A_d^U)^{T_\lambda}=\bigoplus_{d\geq0}B_d,
	\end{displaymath}
	which is finitely generated by \cite[Thm. A]{Gro97} as $T^\lambda$ is clearly reductive and acts rationally.
	As the actions of $S_p$ and $\GL(V)$ commute, and we only take invariants with respect to the $\GL(V)$-action,
	$S_p$ still acts on $\bigoplus_{d\geq0}B_d$.
	We now look at the graded pieces of this algebra.
	
	By Schur-Weyl duality we have
	\begin{displaymath}
		A_d=\bigoplus\limits_{\mu\vdash p}V_\mu\otimes\mathbb{S}_\mu(S^{dk}(V))
	\end{displaymath}
	as a $S_p\times\GL(V)$ representation.
	Taking $U$-invariants yields
	\begin{displaymath}
		A_d^U=\bigoplus\limits_{\mu\vdash p}V_\mu\otimes\mathbb{S}_\mu(S^{dk}(V))^U,
	\end{displaymath}
	as $\GL(V)$ only acts on the plethysms $\mathbb{S}_\mu(S^{dk}(V))$.
	By \cite[p. 115]{Fu97} a vector $v\in A_d$ for $\mu\vdash p$ is a highest weight vector if and only if it is $U$-invariant and a weight vector.
	As we have a weight space decomposition,
	this implies that $A_d^U$ and $\mathbb{S}_\mu(S^{dk}(V))^U$ are the spaces spanned by respective highest weight vectors.
	
	Now let $v\in A_d^U$ be a highest weight vector with weight $\pi\vdash pdk$, $l(\pi)\leq p$,
	where $\pi\neq d\lambda$.
	We might assume $\dim(V)\geq 2$.
	
	If $\lambda=(pk)$, then $t:=(1,2,1,\ldots,1)\in T_\lambda$ and $t\cdot v=2^{\pi_2}v\neq v$, as $\pi\neq d\lambda=(pdk)$ and therefore $\pi_2>0$.
	
	Otherwise, since $\abs{d\lambda}=\abs{\pi}$,
	we can assume $\pi_1>d\lambda_1$ and $\pi_2<d\lambda_2$ without loss of generality.
	Then 
	\begin{displaymath}
		t:=(2^{\frac{1}{\lambda_1}},2^{-\frac{1}{\lambda_2}},1,\ldots,1)\in T_\lambda
	\end{displaymath}
	and
	\begin{displaymath}
		t\cdot v=2^{\frac{\pi_1}{\lambda_1}-\frac{\pi_2}{\lambda_2}}v\neq v,
	\end{displaymath}
	since $\frac{\pi_1}{\lambda_1}-\frac{\pi_2}{\lambda_2}=d\left(\frac{\pi_1}{d\lambda_1}-\frac{\pi_2}{d\lambda_2}\right)>d(1-1)=0$.
	
	As highest weight vectors of weight $d\lambda$ clearly are invariant under $T_\lambda$,
	we get that $B_d$ consist of exactly the highest weight vectors of weight $d\lambda$,
	implying
	\begin{displaymath}
		c_{p,dk}^{d\lambda}=\dim((A_d^U)^{T_\lambda}).
	\end{displaymath}
	Also, passing to $T_\lambda$-invariants yields
	\begin{displaymath}
		B_d=(A_d^U)^{T_\lambda}=\bigoplus\limits_{\mu\vdash p}V_\mu\otimes(\mathbb{S}_\mu(S^{dk}(V))^U)^{T_\lambda}.
	\end{displaymath}
	Thereby, the multiplicity of $V_\mu$ in $B_d$ for any $\mu\vdash p$ is $\dim(\mathbb{S}_\mu(S^{dk}(V))^U)^{T_\lambda})$,
	and again we deduce
	\begin{displaymath}
		a_{\mu,(dk)}^{d\lambda}=
		\dim(\mathbb{S}_\mu(S^{dk}(V))^U)^{T_\lambda}).
	\end{displaymath}
\end{proof}
In general,
if $C=\bigoplus_{d\geq0}C_d$ is a finitely generated, graded algebra, then $\dim(C_d)$ is a quasi-polynomial in $d$,
partially recovering \cref{tensorasymp:cor2}. 
We can use the algebra structure to improve \cref{tensorasymp:cor2}.
\begin{corollary}\label{connection:cor1}
	The function $c_{p,dk}^{d\lambda}$ is a non-negative, non-decreasing quasi-polynomial of positive degree with constant leading term,
	if $\lambda$ is not of the form
	\begin{displaymath}
		(pk),(k^{p}),(a^{p-1}),(b,c^{p-1}),(b^{p-1},c),
	\end{displaymath}
	and otherwise constantly 1.
\end{corollary}
\begin{proof}
	By \cref{tensorasymp:cor2} $c_{p,dk}^{d\lambda}$ is a quasi-polynomial of positive degree,
	and if $c_{p,dk}^{d\lambda}$ is non-decreasing in $d$ its leading term clearly has to be constant,
	so it is enough to show that $c_{dk,p}^{d\lambda}$ is non-decreasing.
	
	By \cref{tensorasymp:lemma1} $c_{p,k}^{\lambda}\geq 1$,
	and hence $\dim((A_1^U)^{T_\lambda})\geq1$ by the above proposition \cref{connection:prop1}.
	So let $v\in(A_1^U)^{T_\lambda}$, $v\neq 0$.
	Since $\bigoplus_{d\geq0}(A_d^U)^{T_\lambda}$ is a graded algebra, $v\cdot(A_d^U)^{T_\lambda}\subset(A_{d+1}^U)^{T_\lambda}$.
	Using \cref{connection:prop1} this yields
	\begin{displaymath}
		c_{p,dk}^{d\lambda}=
		\dim((A_d^U)^{T_\lambda})=
		\dim(v\cdot(A_d^U)^{T_\lambda})\leq
		\dim((A_{d+1}^U)^{T_\lambda})=c_{p,(d+1)k}^{(d+1)\lambda},
	\end{displaymath}
	as $\bigoplus_{d\geq0}A_d$ is an integral domain.
\end{proof}
We now can use the following slight modification of \cite{How89}, which gives us a direct path to deducing asymptotics of $a_{p,(dk)}^{d\lambda}$ in $d$.
\begin{theorem}\label{connection:Howe}
	Let $\beta:S_p\to\Aut(\bigoplus_{d\geq0}B_d)$ be the group homomorphism giving rise to representations $\beta_d:S_p\to\GL(B_d)$ for each $d\geq 0$ as in \cref{connection:prop1}, and define
	\begin{displaymath}
		PK:=\{\sigma\in S_p:\forall d\geq0\ \exists c\in\C^*:\beta_d(\sigma)=c\cdot \id\}.
	\end{displaymath}
	Then, if $PK=\{1\}$, we have
	\begin{displaymath}
		\lim\limits_{d\to\infty}\frac{f_\mu(d)}{\dim(B_d)}=\frac{\dim(V_\mu)}{p!}
	\end{displaymath}
	 for any $\mu\vdash p$, where $f_\mu(d)$ is defined as the multiplicity of $V_\mu$ in $B_d$.
\end{theorem}
\begin{proof}
	Assume $PK=\{1\}$,
	and let $\beta_d:S_p\to\GL(B_d)$ denote the group homomorphism with $\beta(\sigma)|_{B_d}=\beta_d(\sigma)$ for any $\sigma\in S_p$.
	Furthermore, let $\mu\vdash p$.

	By \cite[Cor. 2.16]{FH04} the multiplicity of $V_\mu$ in $B_d$ is
	\begin{displaymath}
		f_\mu(d)=\frac{1}{p!}\sum\limits_{\sigma\in S_p}\trace(\beta_d(\sigma))\chi_\mu(\sigma)=
		\frac{\dim(V_\mu)}{p!}\dim(B_d)+\frac{1}{p!}\sum\limits_{\sigma\in S_p,\sigma\neq\id}\trace(\beta_d(\sigma))\chi_\mu(\sigma),
	\end{displaymath}
	where $\chi_\mu$ denotes the character of $V_\mu$.
	Therefore,
	\begin{displaymath}
		\frac{f_\mu(d)}{\dim(B_d)}=
		\frac{\dim(V_\mu)}{p!}+\frac{1}{p!}\sum\limits_{\sigma\in S_p,\sigma\neq\id}\frac{\trace(\beta_d(\sigma))}{\dim(B_d)}\chi_\mu(\sigma),
	\end{displaymath}
	so proving
	\begin{displaymath}
		\lim\limits_{d\to\infty}\frac{\trace(\beta_d(\sigma))}{\dim(B_d)}=0
	\end{displaymath}
	for $\sigma\neq\id$ proves the claim.
	So fix $\sigma\in S_p$, $\sigma\neq\id$, and let $o(\sigma)$ denote the order of $\sigma$, i.e., the minimal $k\in\N$ such that $\sigma^k=\id$.
	As
	\begin{displaymath}
		\beta_d(\sigma)^{o(\sigma)}=
		\beta_d(\sigma^{o(\sigma)})=\beta_d(\id)=\id,
	\end{displaymath}
	the possible eigenvalues of $\beta_d(\sigma)$ for any $d\geq 0$ are the $o(\sigma)$-th roots of unity $\omega$, of whom there are $o(\sigma)$.
	Let $B_d(\sigma,\omega)$ be the space of eigenvectors of $\beta_d(\sigma)$ with eigenvalue $\omega$, so
	\begin{displaymath}
		B_d=\bigoplus\limits_{\omega}B_d(\sigma,\omega),
	\end{displaymath}
	as $\beta_d(\sigma)$ is of finite order in $\GL(B_d)$ and hence diagonalizable.
	Furthermore, for $v_1\in B_{d_1}(\sigma,\omega_1)$ and $v_2\in B_{d_2}(\sigma,\omega_2)$,
	where $\omega_1$ and $\omega_2$ are $o(\sigma)$-th roots of unity,
	\begin{displaymath}
		\beta_{d_1+d_2}(\sigma)(v_1\cdot v_2)=
		\beta(\sigma)(v_1)\cdot\beta(\sigma)(v_2)=
		\beta_{d_1}(\sigma)(v_1)\cdot\beta_{d_2}(\sigma)(v_2)=
		\omega_1\omega_2v_1\cdot v_2,	
	\end{displaymath}
	i.e., $v_1\cdot v_2\in B_{d_1+d_2}(\sigma,\omega_1\omega_2)$.
	As $\bigoplus_{d\geq0}B_d$ is an integral domain,
	this implies that if $\omega_1$ is an eigenvalue of $\beta_{d_1}(\sigma)$ and $\omega_2$ is an eigenvalue of $\beta_{d_2}(\sigma)$, then $\omega_1\omega_2$ is an eigenvalue of $\beta_{d_1+d_2}(\sigma)$, i.e.,
	\begin{equation}\label{connection:eq1}
		R_{d_1}(\sigma)\cdot R_{d_2}(\sigma)\subset R_{d_1+d_2}(\sigma),
	\end{equation}
	where we denote the set of eigenvalues of $\beta_d(\sigma)$ by $R_d(\sigma)$, the spectrum of $\beta_d(\sigma)$.
	Now let $\omega_1,\omega_2\in R_d(\sigma)$.
	Then, for any $0\leq j\leq o(\sigma)-1$,
	\begin{displaymath}
		(\omega_1\omega_2^{-1})^j=
		\omega_1^j\omega_2^{o(\sigma)-j}\in R_{o(\sigma)d}(\sigma)
	\end{displaymath}
	by \cref{connection:eq1}.
	As $\omega_1^{o(\sigma)},\omega_2^{o(\sigma)}=1$ we also have $(\omega_1\omega_2^{-1})^{o(\sigma)}=1$,
	and therefore $R_{o(\sigma)d}(\sigma)$ contains the group generated by $\omega_1\omega_2^{-1}$.
	
	Furthermore, if $R_{d_1}(\sigma)$ contains the group $G_1$ and $R_{d_2}$ contains the group $G_2$,
	then by \cref{connection:eq1} $R_{d_1+d_2}(\sigma)$ contains the group $G_1G_2$ generated by $G_1$ and $G_2$.
	As $\omega_1\omega_2^{-1}$ for any $o(\sigma)$-th roots of unity $\omega_1,\omega_2$ can only attain finitely man values,
	each of the values $\omega_1\omega_2^{-1}$ for any $d\geq 0$ and $\omega_1,\omega_2\in R_d(\sigma)$ must have been attained by $\omega_1\omega_2^{-1}$ for $\omega_1,\omega_2\in R_d(\sigma)$ where $N\geq d\geq 0$ for some $N\in\N$.
	The above arguments now show that there is some $d_0\geq0$ such that $R_{d_0}(\sigma)$ contains the group generated by all ratios $\omega_1\omega_2^{-1}$,
	where $\omega_1,\omega_2\in R_d(\sigma)$ for some $d\geq 0$.
	Fix such a $d_0$.
	We claim that $R_{d_0}(\sigma)$ is the full group of $o(\sigma)$-th roots of unity.
	
	Indeed, the group $G_\sigma$ generated by all ratios $\omega_1\omega_2^{-1}$,
	where $\omega_1,\omega_2\in R_d(\sigma)$ for some $d\geq 0$,
	is a subgroup of the group of $o(\sigma)$-th roots of unity.
	Therefore, $r:=\abs{G_\sigma}\leq o(\sigma)$ and $r$ divides $o(\sigma)$.
	Furthermore, the eigenvalues of $\beta_d(\sigma^r)$ for some $d\geq 0$ are the $r$-th powers of eigenvalues of $\beta_d(\sigma)$,
	as $\beta_d(\sigma^r)=\beta_d(\sigma)^r$.
	If now $\omega_1,\omega_2\in R_{d}(\sigma^r)$,
	then there are $\tilde\omega_1,\tilde\omega_2\in R_d(\sigma)$ with $\omega_1=\tilde\omega_1^r$, $\omega_2=\tilde\omega_2^r$,
	and therefore
	\begin{displaymath}
		\omega_1\omega_2^{-1}=
		\tilde\omega_1^r\tilde\omega_2^{-r}=
		(\tilde\omega_1\tilde\omega_2^{-1})^r=1,
	\end{displaymath}
	as $\tilde\omega_1\tilde\omega_2^{-1}\in G_\sigma$ and $\abs{G_\sigma}=r$.
	But this implies that $\beta_d(\sigma^r)$ is just multiplication by a scalar for any $d\geq 0$, and therefore $\sigma^r\in PK$.
	As $PK=\{1\}$, we conclude $\sigma^r=\id$ and hence $\abs{G_\sigma}=r=o(\sigma)$,
	so $G_\sigma$ is the full group of $o(\sigma)$-th roots of unity.

	Furthermore, for any $d\geq 0$ and $o(\sigma)$-th roots of unity $\omega_1$ and $\omega$ we have
	\begin{displaymath}
		B_{d}(\sigma,\omega_1)\cdot B_{d_0}(\sigma,\omega_1^{-1}\omega)\subset B_{d+d_0}(\sigma,\omega),
	\end{displaymath}
	and therefore, as $\bigoplus_{d\geq0}B_d$ is an integral domain and $B_{d_0}(\sigma,\omega_1^{-1}\omega)\neq\{0\}$,
	\begin{displaymath}
		\dim(B_{d}(\sigma,\omega_1))\leq
		\dim(B_{d+d_0}(\sigma,\omega)).
	\end{displaymath}
	Since $B_{d}=\bigoplus_{\omega_1}B_{d}(\sigma,\omega_1)$,
	where the direct sum ranges over all $o(\sigma)$-th roots of unity $\omega_1$,
	we conclude
	\begin{displaymath}
		\dim(B_{d})\leq
		o(\sigma)\dim(B_{d+d_0}(\sigma,\omega))
	\end{displaymath}
	by summing over all $o(\sigma)$ many $o(\sigma)$-th roots of unity $\omega_1$.
	By \cref{connection:cor1} and \cref{connection:prop1} $\dim(B_d)$ is either a quasi-polynomial of positive degree in $d$ with constant leading term or constantly 1,
	from which we conclude
	\begin{displaymath}
		\lim\limits_{d\to\infty}\frac{\dim(B_{d})}{\dim(B_{d+d_0})}=1,
	\end{displaymath}
	which combined with
	$\dim(B_{d})\leq
		o(\sigma)\dim(B_{d+d_0}(\sigma,\omega))$ for any $d\geq0$ and $o(\sigma)$-th root of unity $\omega$ yields
	\begin{displaymath}
		\liminf_{d\to\infty}\frac{\dim(B_d(\sigma,\omega))}{\dim(B_d)}\geq\frac{1}{o(\sigma)}.
	\end{displaymath}
	But $B_d=\bigoplus_{\omega}B_d(\sigma,\omega)$, where the sum ranges over all $o(\sigma)$ many $o(\sigma)$-th roots of unity $\omega$,
	so we must already have
	\begin{displaymath}
		\lim\limits_{d\to\infty}\frac{B_d(\sigma,\omega)}{\dim(B_d)}=\frac{1}{o(\sigma)}
	\end{displaymath}
	for any $o(\sigma)$-th root of unity $\omega$.
	With this we finally deduce
	\begin{displaymath}
		\lim\limits_{d\to\infty}\frac{\trace(\beta_d)(\sigma)}{\dim(B_d)}=
		\lim\limits_{d\to\infty}\sum\limits_{\omega}\frac{\omega\dim(B_d(\sigma,\omega))}{\dim(B_d)}=
		\sum\limits_{\omega}\omega\lim\limits_{d\to\infty}\frac{\dim(B_d(\sigma,\omega))}{\dim(B_d)}=
		\frac{1}{o(\sigma)}\sum\limits_{\omega}\omega=
		0,
	\end{displaymath}
	where the sums ranges over all $o(\sigma)$-th roots of unity $\omega$.
\end{proof}
Using this theorem and the \cref{connection:prop1} we can now deduce asymptotics of the plethysm coefficients $a_{\mu,(dk)}^{d\lambda}$, if $PK=\{1\}$.
This gets even easier if we use basic properties of the symmetric group.
We denote the alternating group $A_p:=\{\sigma\in S_p:\sign(\sigma)=1\}\subset S_p$ by $A_p$.
Then, for $p\neq 4$, $A_p$ is the only non-trivial normal subgroup of $S_p$, and for $p=4$ the non-trivial normal subgroups are $V\subset A_4$, where $V$ is the Klein four group.
\begin{lemma}\label{connection:cor3}
The only irreducible representations of $S_p$ on which each element of $A_p$ acts as a scalar are the trivial representation $S_p\to\C^*,\sigma\mapsto1$ and the sign representation $S_p\to\C^*,\sigma\mapsto\sign(\sigma)$.
\end{lemma}
\begin{proof}
	First, let $p\neq 4$.
	Let $\rho:S_p\to\GL(V)$ be an irreducible representation of $S_p$ such that for each $\sigma\in A_p$ there is an $c_\sigma\in\C^*$ with $\rho(\sigma)=c_\sigma\id$.
	
	If now $\sigma,\tau\in A_p$ with $\rho(\sigma)=c_\sigma\id$ and $\rho(\tau)=c_\tau\id$,
	then
	\begin{displaymath}
		\rho(\sigma\tau\sigma^{-1}\tau^{-1})=
		\rho(\sigma)\rho(\tau)\rho(\sigma^{-1})\tau(\sigma)^{-1}=
		c_\sigma c_\tau c_\sigma^{-1} c_\tau^{-1} \id=\id.
	\end{displaymath}
	As elements of this form generate $A_p$,
	we conclude that $\rho(A_p)=\{\id\}$.
	Therefore, we get a representation
	\begin{displaymath}
		[\rho]:S_p/A_p\to\GL(V),[\sigma]\to\rho(\sigma),
	\end{displaymath}
	where $[\sigma]$ denotes the residue class $\sigma A_p$ of $\sigma$ in $S_p/A_p$.
	
	Furthermore,
	$[\sign]$ and $[\mathrm{trivial}]$,
	where $\sign$ and $\mathrm{trivial}$ denote the $\sign$ representation and trivial representation of $S_p$ respectively,
	are distinct irreducible representations of $S_p/A_p$, both of degree 1,
	and
	\begin{displaymath}
		\abs{S_p/A_p}=2=1^2+1^2.
	\end{displaymath}
	Thereby, \cite[Cor. 2.18]{FH04} shows that these are the only irreducible representations of $S_p/A_p$.
		
	Lastly, if $p=4$ one can just list all 5 irreducible presentations of $S_4$ and note that exactly the trivial and sign representations are those for which $A_4$ acts as a scalar.
\end{proof}
As the group $PK$ from \cref{connection:Howe} is clearly normal, either $PK=\{1\}$ or $A_p\subset PK$.
In the latter case, all vectors in $B_d$ for any $d\geq 0$ are a sum of symmetric and skew-symmetric vectors by \cref{connection:cor3}.
Hence, we only have to find a highest weight vector which is not a sum of a symmetric and skew-symmetric vector to show $PK=\{1\}$ and thereby deducing results on the asymptotics of $a_{\mu,(dk)}^{d\lambda}$ for any $\mu\vdash p$.
This is exactly what we are concerned with in the next section.
\subsection{Constructing highest weight vectors, final result}
From now on, we assume that a basis $v_1,\ldots,v_n$ of $V$ is fixed, such that elements of the form
\begin{displaymath}
	x^{\alpha_1}\otimes\ldots\otimes x^{\alpha_p},
\end{displaymath}
where $\alpha_1,\ldots,\alpha_p\in\N_0^n$, $\abs{\alpha_1}=\ldots=\abs{\alpha_p}=k$, form a basis of $(S^k(V))^{\otimes p}$.
In the following, when talking about coefficients of these basis vectors, we shall mean coefficients with respect to this basis.
\begin{definition}
	Let $T$ be a SSYT of some shape $\pi$ with $l(\pi)\leq p$, filled with $1,\ldots,p$, $\mu:=\pi^T$,
	and let $j_T(a,b)$ for $1\leq a\leq l(\lambda)$ and $1\leq b\leq \lambda_a$ denote the entry of $T$ of the box in the $a$-th row and $b$-th column.
	We now define
	\begin{displaymath}
		j_T:=(j_T(l(\lambda),\lambda_{l(\lambda)}),\ldots,j_T(l(\lambda),1),\ldots,j_T(1,\lambda_1),\ldots,j_T(1,1)),
	\end{displaymath}
	i.e., $j_T$ is the vector we obtain from $T$ by reading entries right to left and bottom to top,
	or by ordering the entries $j_T(a,b)$ according to the position $(a,b)$ decreasingly in the lexicographic order.
	Furthermore,
	we define
	\begin{displaymath}
		h_T:=
		\sum\limits_{\sigma_1\in S_{\mu_1},\ldots,\sigma_{\lambda_1}\in S_{\mu_{\lambda_1}}}\sign(\sigma_1)\ldots\sign(\sigma_{\lambda_1})
		\bigotimes\limits_{i=1}^p\prod\limits_{(a,b):j_T(a,b)=i}x_{\sigma_b(a)}\in S^{i_1}(V)\otimes\ldots\otimes S^{i_p}(V),
	\end{displaymath}
	where $i_m$ for $1\leq m\leq p$ is the number of entries of $T$ equal to $m$.
\end{definition}
	We are going to show that for each SSYT of shape $\lambda$ filled with $k$ 1's, $\ldots$, $p$'s $h_T$ is a highest weight vector of weight $\lambda$.
	In fact, they form a basis of the space of highest weight vectors in $(S^k(V))^{\otimes p}$ of weight $\lambda$.
\begin{example}
	Consider the following tableau of shape $(4,2)$
	\begin{displaymath}
		T=\young(1123,23).
	\end{displaymath}
	We get
	\begin{displaymath}
		j_T=(3,2,3,2,1,1),
	\end{displaymath}
	and, by \glqq permuting entries of the columns of $T$\grqq,
	\begin{align*}
		h_T=&
		\sign(\id)\sign(\id)x_{\id(1)}x_{\id(1)}\otimes x_{\id(2)}x_1\otimes x_{\id(2)}x_1\\
		&+\sign((12))\sign(\id)x_{(12)(1)}x_{\id(1)}\otimes x_{(12)(2)}x_1\otimes x_{\id(2)}x_1\\
		&+\sign(\id)\sign((12))x_{\id(1)}x_{(12)(1)}\otimes x_{\id(2)}x_1\otimes x_{(12)(2)}x_1\\
		&+\sign((12))\sign((12))x_{(12)(1)}x_{(12)(1)}\otimes x_{(12)(2)}x_1\otimes x_{(12)(2)}x_1\\
		=&x_1^2\otimes x_1x_2\otimes x_1x_2-
		x_1x_2\otimes x_1^2\otimes x_1x_2
		-x_1x_2\otimes x_1x_2\otimes x_1^2
		+x_2^2\otimes x_1^2\otimes x_1^2.
	\end{align*}
\end{example}
\begin{proposition}\label{connection:prop}
	The $h_T$ for $T$ SSYTs of shape $\lambda$ filled with $k$ $1's,\ldots, p's$ form a basis of the space of highest weight vectors of weight $\lambda$ in $(S^k(V))^{\otimes p}$.
\end{proposition}
\begin{proof}
	This works like \cite[Prop. 2.3]{MM14}, where the \glqq dual problem\grqq\  $(\wedge^k(V))^{\otimes p}$ is considered.
\end{proof}
\begin{lemma}\label{connection:lemma4}
	Let $T$ be a SSYT of some shape $\pi$ filled with $1,\ldots,p$, $\mu:=\pi^T$, and let $\sigma_1\in S_{\mu_1},\ldots,\sigma_{\pi_1}\in S_{\mu_{\pi_1}}$ such that
	\begin{displaymath}
		\bigotimes\limits_{i=1}^p\prod\limits_{(a,b):j_T(a,b)=i}x_{\sigma_b(a)}=
		\bigotimes\limits_{i=1}^p\prod\limits_{(a,b):j_T(a,b)=i}x_{a}=:t.
	\end{displaymath}
	Then, $\sigma_1=\id,\ldots,\sigma_{\pi_1}=\id$.
\end{lemma}
\begin{proof}
	We use induction on $\abs{\pi}$.
	For $\pi=(1)$ the claim obviously is true.
	
	So assume for some $N\in\N$ the claim holds for all SSYTs of some shape $\pi'$ filled with $1,\ldots,p$, where $\abs{\pi'}\leq N$,
	and let $T$ be a SSYT filled with $1,\ldots,p$ of shape $\pi\vdash N+1$,
	as well as $\sigma_1\in S_{\mu_1},\ldots,\sigma_{\pi_1}\in S_{\mu_{\pi_1}}$ such that
	\begin{displaymath}
		\bigotimes\limits_{i=1}^p\prod\limits_{(a,b):j_T(a,b)=i}x_{\sigma_b(a)}=
		\bigotimes\limits_{i=1}^p\prod\limits_{(a,b):j_T(a,b)=i}x_{a}.
	\end{displaymath}
	
	Let $k\in\N$ be minimal such that
	\begin{displaymath}
		j_T(l(\pi),\pi_{l(\pi)})=\ldots=j_T(l(\pi),k).
	\end{displaymath}
	As $T$ is a SSYT and by the choice of $k$,
	entries to the left of the $k$-th column are smaller than $j_T(l(\pi),k)$,
	so any of the variables in the $j_T(l(\pi),k)$-th component
	\begin{displaymath}
		\prod\limits_{(a,b): j_T(a,b)=j_T(l(\pi),k)}x_{\sigma_b(a)}
	\end{displaymath}
	must come from $(a,b)$ with $b\geq k$.
	But for $(a,b)$ with $b>l(\pi)$ we have $\sigma_b(a)<l(\pi)$, as $\mu_b<l(\pi)$ for these $(a,b)$,
	and hence each variable $x_{l(\pi)}$ must come from $\sigma_k(l(\pi)),\ldots,\sigma_{\lambda_{l(\pi)}}(l(\pi))$, i.e.,
	\begin{displaymath}
		\sigma_k(l(\pi))=\ldots=\sigma_{\pi_{l(\pi)}}(l(\pi))=l(\pi).
	\end{displaymath}
	Now let $S$ be the SSYT obtained from $T$ by deleting the $k$-th up to $\pi_{l(\pi)}$-th box in the last row.
	As
	\begin{displaymath}
		\sigma_k(l(\pi))=\ldots=\sigma_{\pi_{l(\pi)}}(l(\pi))=l(\pi),
	\end{displaymath}
	we have
	\begin{displaymath}
		\bigotimes_{i=1}^p\prod\limits_{(a,b):j_S(a,b)=i}x_{\sigma_b(a)}=
		\bigotimes_{i=1}^p\prod\limits_{(a,b):j_S(a,b)=i}x_{a}.
	\end{displaymath}
	The induction hypothesis yields $\sigma_1=\id,\ldots,\sigma_{\pi_1}=\id$,
	where we view $\sigma_k,\ldots,\sigma_{\pi_{l(\pi)}}\in S_{l(\pi)-1}$.
	But $\sigma_1(l(\pi))=\ldots=\sigma_{\pi_{l(\pi)}}(l(\pi))=l(\pi)$,
	and therefore $\sigma_1,\ldots,\sigma_{\pi_{l(\pi)}}=\id$ as elements of $S_{l(\pi)}$,
	and the claim follows.
\end{proof}
\begin{remark}
	As elements of the form $x^{\alpha_1}\otimes\ldots\otimes x^{\alpha_p}$ with $\alpha_1,\ldots,\alpha_p\in\N_0^n$, $\abs{\alpha_1}=\ldots=\abs{\alpha_p}=k$ form a basis of $(S^k(V))^{\otimes p}$,
	the above
	 \cref{connection:lemma4} shows that the coefficient w.r.t. this basis of 
	\begin{displaymath}
	\bigotimes\limits_{i=1}^p\prod\limits_{(a,b):j_T(a,b)=i}x_{a}
	\end{displaymath}
	in $h_T$ is 1, so in particular $h_T\neq 0$,
	where $T$ is a SSYT of shape $\lambda$ filled with $k$ 1's, $\ldots$, $p$'s.
\end{remark}
\begin{proposition}\label{connection:prop3}
	If $\lambda$ is not of the form $(pk)$ or $(a^{p-1})$ for some integer $a$,
	$l(\lambda)\leq p-1$, and $p\geq 3$, then there is an even permutation $\sigma\in S_p$ and a tableau $T$ of shape $\lambda$ filled with $k$ 1's $\ldots$, $p$'s such that $h_T\cdot\sigma\neq h_T$.
	Furthermore, if $p=4$ and $\lambda\neq (2k,2k)$, we can choose $\sigma\in V$.
\end{proposition}
\begin{proof}
	Let $\mu:=\lambda^T$.
	
	First, we assume that not all parts of $\lambda$ are equal,
	and choose $i_0$ minimal such that $\lambda_{i_0}>\lambda_{i_0+1}$.
	We have to distinguish a few cases.
	The reader might easily verify that, given the assumptions of each case, all crossing out we perform throughout this proof works.
	
	\underline{Case 1:} Assume $\lambda_{i_0+1}\geq k$.
	
	We choose $j_0$ such that $\lambda_{j_0+1}<k\leq\lambda_{j_0}$.
	We then for each $j_0<j\leq l(\lambda)$ cross out the rightmost $\lambda_{j}-\lambda_{j+1}$ boxes in the $j$-th row,
	and the rightmost $k-\lambda_{j_0+1}$ boxes in the $j_0$-th row.
	By the choice of $j_0$
	we obtain a Young diagram of some shape $\lambda'$, where $\lambda'\vdash (p-1)k$.
	Since $l(\lambda)\leq p-1$, we in particular have $l(\lambda')\leq p-1$.
	
	
	By \cref{tensorasymp:lemma1} we find a SSYT of shape $\lambda'$ filled with $k$ 1's, $\ldots$, $p-1$'s.
	Adding back the boxes we crossed out before and filling them with $p$, we obtain a SSYT $T$ of shape $\lambda$ filled with $k$ 1's, $\ldots$, $p$'s.
	As $\lambda_{i_0+1}\geq k$, all $p$'s are in rows below the $i_0$-th row.
	Furthermore, let $j:=j_T(i_0,\lambda_{i_0})$ be the entry of $T$ in the $i_0$-th row and $\lambda_{i_0}$-th column, and choose any entry $\tilde j$ distinct from both $j$ and $p$.
	We claim that $h_T\cdot (j\ p\ \tilde j)\neq h_T$.
	
	\begin{figure}[H]
		\centering
		\begin{tikzpicture}
			\draw (0,0) -- (2,0);
			\node at (.2,.2) {$p$};
			\node at (1,.2) {$\ldots$};
			\node at (1.8,.2) {$p$};
			\draw (2,0) -- (2,.5);
			\draw (2,.5)--(5,.5);
			\node at (3.2,.7) {$p$};
			\node at (4,.7) {$\ldots$};
			\node at (4.8,.7) {$p$};
			\draw (5,.5)--(5,1);
			\draw[decorate,decoration={brace,amplitude=10pt}] (7,1)--(5,1) node[midway,yshift=-17pt]{no $p$};
			\draw[decorate,decoration={brace,amplitude=10pt}] (7,2)--(7,1) node[midway,xshift=32pt]{$i_0$ rows};
			\draw (5,1)--(7,1);
			\node at (6.8,1.2) {$j$};
			\draw (7,1)--(7,2);
			\draw (7,2)--(0,2);
			\draw (0,2)--(0,0);
			\draw[dashed] (5,1)--(5,2);
		\end{tikzpicture}
		\caption*{schematic picture of $T$}
	\end{figure}

	Indeed, for any $\sigma_1\in S_{\mu_1}, \ldots, \sigma_{\lambda_1}\in S_{\mu_{\lambda_1}}$ in
	\begin{displaymath}
		\prod\limits_{(a,b):j_T(a,b)=j}x_{\sigma_b(a)}
	\end{displaymath}
	at least one of the variables $x_1,\ldots,x_{i_0}$ must appear, as $j_T(i_0,\lambda_{i_0})=j$ and $\sigma_{\lambda_{i_0}}(i_0)\leq i_0$ because of $\sigma_{\lambda_{i_0}}\in S_{\mu_{\lambda_{i_0}}}$ and $\mu_{\lambda_{i_0}}=i_0$.
	Therefore,
	if we look at the coefficients of basis elements $x^{\alpha_1}\otimes\ldots\otimes x^{\alpha_p}$ in $h_T$, where $\alpha_1,\ldots,\alpha_p\in\N_0^n$, $\abs{\alpha_1}=\ldots=\abs{\alpha_p}=k$,
	these can be non-zero only if in $x^{\alpha_j}$ one of the variables $x_1,\ldots,x_{i_0}$ appears.
	
	But in
	\begin{displaymath}
		\prod\limits_{(a,b):j_T(a,b)=p}x_a
	\end{displaymath}
	only $x_{i_0+1},\ldots,x_{l(\lambda)}$ appear, as all $p$ are in rows below the $i_0$-th row of $T$.
	Since the coefficient of the basis element
	\begin{displaymath}
		\bigotimes\limits_{i=1}^p\prod\limits_{(a,b):j_T(a,b)=i}x_{a}
	\end{displaymath}
	in $h_T$ is 1 by \cref{connection:lemma4},
	the coefficient of
	\begin{displaymath}
		\left(\bigotimes\limits_{i=1}^p\prod\limits_{(a,b):j_T(a,b)=i}x_{a}\right)\cdot (j\ p\  \tilde j)
	\end{displaymath}
	in $h_T\cdot (j\ p\ \tilde j)$  is 1,
	and in the $j$-th component none of $x_1,\ldots,x_{i_0}$ appear.
	Therefore, $h_T\cdot (j\ \tilde j\ p)\neq h_T$.
	If $p=4$, simply take $(j\ 4)(j'\ \tilde j)$ for $j,j'\neq p,j$ distinct.
	
	\underline{Case 2:} Assume $\lambda_{i_0+1}<k$ and $i_0\geq 2$.
	As $\lambda\vdash pk$ and $l(\lambda)\leq p$,
	the choice of $i_0$ implies $\lambda_{i_0}>k$.
	We then for each $i_0<j\leq l(\lambda)$ cross out the rightmost $\lambda_{j}-\lambda_{j+1}$ boxes in the $j$-th row,
	and the rightmost $k-\lambda_{i_0+1}$ boxes in the $i_0$-th row.
	As $\lambda_{i_0+1}<k<\lambda_{i_0}$,
	we obtain a Young diagram of some shape $\lambda'$, where $\lambda'\vdash (p-1)k$.
	Furthermore, $l(\lambda')\leq p-2$, as $\lambda_{i_0+1}<k$ and in particular $\lambda_{l(\lambda)}<k$, so that we have crossed out all boxes in the last row.

	Afterwards, we choose $j_1$ such that $\lambda'{}_1-\lambda'{}_{j_1+1}\geq k>\lambda'{}_1-\lambda'{}_{j_1}$.
	We then for each $1\leq j\leq j_1$ cross out the $\lambda'{}_j-\lambda'{}_{j+1}$ rightmost boxes in the $j$-th row,
	and the rightmost $k-(\lambda'{}_1-\lambda'{}_{j_1})$ boxes in the $j_1$-th row of the Young diagram of shape $\lambda'$.
	By the choice of $j_1$ we obtain a Young diagram of some shape $\lambda''$ where $\lambda''\vdash (p-2)k$.

	By \cref{tensorasymp:lemma1} we find a SSYT of shape $\lambda''$ filled with $k$ 1's, $\ldots$, $p-2$'s.
	Adding back the $k$ boxes we crossed out in the second step and filling them with $p-1$, we obtain a SSYT of shape $\lambda'$ filled with $k$ 1's, $\ldots$, $p-1$'s,
	and then adding back the $k$ boxes we crossed out first and filling them with $p$, we obtain a SSYT $T$ of shape $\lambda$ filled with $k$ 1's, $\ldots$, $p$'s.
	By the construction of $T$ there are more $p-1$'s than $p$'s in columns to the right of the $\lambda_{i_0+1}$-th column.
	We claim that $h_T\cdot (1\ p-1\ p)\neq h_T$.

	\begin{figure}[H]
		\centering
		\begin{tikzpicture}
			\draw (0,0) -- (2,0);
			\node at (.2,.2) {$p$};
			\node at (1,.2) {$\ldots$};
			\node at (1.8,.2) {$p$};
			\draw (2,0) -- (2,.5);
			\draw (2,.5)--(5,.5);
			\node at (2.2,.7) {$p$};
			\node at (3.5,.7) {$\ldots$};
			\node at (4.8,.7) {$p$};
			\draw (5,.5)--(5,1);
			\draw[decorate,decoration={brace,amplitude=10pt}] (9,1)--(5,1) node[midway,yshift=-17pt]{more $\tilde p$ than $p$};
			\draw[decorate,decoration={brace,amplitude=10pt}] (9,2)--(9,1) node[midway,xshift=32pt]{$i_0$ rows};
			\draw (5,1)--(9,1);
			\node at (8.8,1.2) {$p$};
			\node at (8,1.2) {$\ldots$};
			\node at (7.2,1.2) {$p$};
			\node at (8.8,1.7) {$\tilde p$};
			\node at (8,1.7) {$\ldots$};
			\node at (7.2,1.7) {$\tilde p$};
			\node at (6.8,1.2) {$\tilde p$};
			\node at (6,1.2) {$\ldots$};
			\node at (5.2,1.2) {$\tilde p$};
			\draw (9,1)--(9,2);
			\draw (9,2)--(0,2);
			\draw (0,2)--(0,0);
			\draw[dashed] (5,1)--(5,2);
		\end{tikzpicture}
		\caption*{schematic picture of $T$, $\tilde p:=p-1$}
	\end{figure}
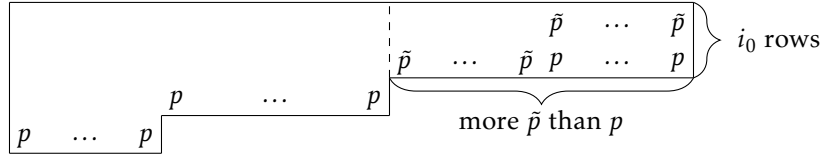

	Indeed, for any $\sigma_1\in S_{\mu_1}, \ldots, \sigma_{\lambda_1}\in S_{\mu_{\lambda_1}}$ in
	\begin{displaymath}
		\prod\limits_{(a,b):j_T(a,b)=p-1}x_{\sigma_b(a)}
	\end{displaymath}
	at least as many of the variables $x_1,\ldots,x_{i_0}$ appear as there are $p-1$'s to the right of the $\lambda_{i_0+1}$-th column of $T$.
	Therefore,
	if we look at coefficients of basis elements $x^{\alpha_1}\otimes\ldots\otimes x^{\alpha_p}$ in $h_T$, where $\alpha_1,\ldots,\alpha_p\in\N_0^n$, $\abs{\alpha_1}=\ldots=\abs{\alpha_p}=k$,
	these can be non-zero only if in $x^{\alpha_{p-1}}$ at least as many of the variables $x_1,\ldots,x_{i_0}$ appear as there are $p-1$'s to the right of the $\lambda_{i_0+1}$-th column of $T$.
	
	But in
	\begin{displaymath}
		\prod\limits_{(a,b):j_T(a,b)=p}x_a
	\end{displaymath}
	only as many $x_{1},\ldots,x_{i_0}$ appear as there are $p$'s to the right of the $\lambda_{i_0+1}$-th column of $T$,
	so in particular fewer then there are $p-1$'s to the right of the $\lambda_{i_0+1}$-th column.
	Since the coefficient of the basis element
	\begin{displaymath}
		\bigotimes\limits_{i=1}^p\prod\limits_{(a,b):j_T(a,b)=i}x_{a}
	\end{displaymath}
	in $h_T$ is 1 by \cref{connection:lemma4},
	the coefficient of
	\begin{displaymath}
		\left(\bigotimes\limits_{i=1}^p\prod\limits_{(a,b):j_T(a,b)=i}x_{a}\right)\cdot (1\ p-1\ p)
	\end{displaymath}
	in $h_T\cdot (1\ p-1\ p)$  is 1,
	and in the $p-1$-th component fewer $x_1,\ldots,x_{i_0}$ appear than there are $p-1$'s to the right of the $\lambda_{i_0+1}$-th column of $T$.
	Therefore, $h_T\cdot (1\ p-1\ p)\neq h_T$.
	If $p=4$, simply take $\sigma=(1\ 2)(3\ 4)$.
	
	\underline{Case 3:} Assume $i_0=1$ and $\lambda_2<k$.
	Then, $\lambda_1\geq 2k$,
	as otherwise $l(\lambda)\leq p-1$ would imply $\abs{\lambda}\leq\lambda_1+(p-2)\lambda_2< 2k+(p-2)k=pk$.
	In particular, we have $\lambda_1-\lambda_2>k$.
	Then, we cross out the rightmost $k$ boxes in the first row,
	and obtain a Young diagram of some shape $\lambda'$ with $\lambda'\vdash (p-1)k$, $l(\lambda')\leq p-1$.
	By \cref{tensorasymp:lemma1} we find a SSYT of shape $\lambda'$ filled with $k$ 1's, $\ldots$, $p-1$'s.
	Adding back the boxes we crossed out before and filling them with $p$, we obtain a SSYT $T$ of shape $\lambda$ filled with $k$ 1's, $\ldots$, $p$'s.
	Let $j:=j_T(2,\lambda_2)$ be the entry of $T$ in the second row and $\lambda_{2}$-th column,
	and choose an entry $\tilde j$ distinct from both $j$ and $p$.
	Arguing similarly as in the preceding cases,
	we see that $h_T\cdot(p\ j\ \tilde j)\neq h_T$, and in case $p=4$ one can simply take $\sigma=(j\ p)(j'\ j'')$ for distinct $j,j'\neq j,p$.
	\begin{figure}[H]
		\centering
		\begin{tikzpicture}
			\draw (4,1)--(5,1);
			\draw (5,1)--(5,1.5);
			\node at (4.8,1.2) {$j$};
			\draw[decorate,decoration={brace,amplitude=10pt}] (9,1.5)--(5,1.5) node[midway,yshift=-17pt]{$k$ many $p$'s};
			\draw (5,1.5)--(9,1.5);
			\draw (9,1.5)--(9,2);
			\draw (9,2)--(4,2);
			\node at (8.8,1.7) {$p$};
			\node at (6.2,1.7) {$p$};
			\node at (7.5,1.7) {$\ldots$};
		\end{tikzpicture}
		\caption*{schematic picture of $T$}
	\end{figure}
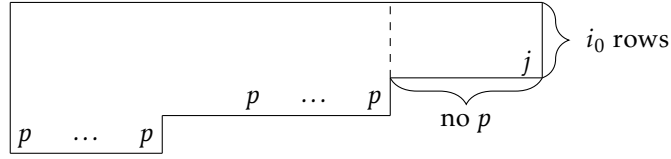
	
	\underline{Case 4:} Lastly, assume all parts of $\lambda$ are equal and $1<l(\lambda)<p-1$.
	Note that for $p=4$ this forces $\lambda=(2k,2k)$, which is exactly the partition we exclude.
	We now cross out the rightmost $k$ boxes in both the $l(\lambda)$-th and $l(\lambda)-1$-th column of the Young diagram of shape $\lambda$,
	 obtaining a Young diagram of some shape $\lambda'\vdash (p-2)k$, $l(\lambda')\leq p-2$.
	By \cref{tensorasymp:lemma1} we find a SSYT of shape $\lambda'$ filled with $k$ 1's, $\ldots$, $p-2$'s.
	We then add back all boxes, and fill those in the $l(\lambda)$-th column with $p$ and those in the $l(\lambda)-1$-th column with $p-1$.
	As $\lambda_1>k$,
	since otherwise $\abs{\lambda}=l(\lambda)\lambda_1\leq(p-2)k<pk$,
	$j:=j_T(l(\lambda),1)$ is neither $p-1$ nor $p$.
	We claim that $h_T\cdot(p-1\ p\ j)\neq h_T$.
	\begin{figure}[H]
		\centering
		\begin{tikzpicture}
			\draw (0,0)--(6,0);
			\draw (1,0)--(1,1);
			\draw (1,.5)--(6,.5);
			\draw (1,1)--(6,1);
			\draw (6,0)--(6,1.5);
			\draw (6,1.5)--(0,1.5);
			\draw (0,1.5)--(0,0);
			\node at (.2,.2) {$j$};
			\node at (3.5,.2) {$k$ many $p$'s};
			\node at (3.5,.7) {$k$ many $p-1$'s};
		\end{tikzpicture}
		\caption*{schematic picture of $T$}
	\end{figure}
	
	Indeed, for any $\sigma_1\in S_{\mu_1}, \ldots, \sigma_{\lambda_1}\in S_{\mu_{\lambda_1}}$ in
	\begin{displaymath}
		\prod\limits_{(a,b):j_T(a,b)=p-1}x_{\sigma_b(a)}\prod\limits_{(a,b):j_T(a,b)=p}x_{\sigma_b(a)}
	\end{displaymath}
	at most $k$ many $x_{l(\lambda)}$ appear.
	Therefore,
	if we look at coefficients of basis elements $x^{\alpha_1}\otimes\ldots\otimes x^{\alpha_p}$ in $h_T$, where $\alpha_1,\ldots,\alpha_p\in\N_0^n$, $\abs{\alpha_1}=\ldots=\abs{\alpha_p}=k$,
	these can be non-zero only if in $x^{\alpha_{p-1}}x^{\alpha_{p}}$ at most $p$ many $x_{l(\lambda)}$ appear.
	
	But in
	\begin{displaymath}
		\prod\limits_{(a,b):j_T(a,b)=j}x_a\prod\limits_{(a,b):j_T(a,b)=p}x_a
	\end{displaymath}
	at least $k+1$ many $x_{l(\lambda)}$ appear.
	Since the coefficient of the basis element
	\begin{displaymath}
		\bigotimes\limits_{i=1}^p\prod\limits_{(a,b):j_T(a,b)=i}x_{a}
	\end{displaymath}
	in $h_T$ is 1 by \cref{connection:lemma4},
	the coefficient of
	\begin{displaymath}
		\left(\bigotimes\limits_{i=1}^p\prod\limits_{(a,b):j_T(a,b)=i}x_{a}\right)\cdot (p-1\ p\ j)
	\end{displaymath}
	in $h_T\cdot(p-1\ p\ j)$  is 1,
	and in the product of the $j$-th and $p$-th component at least $p+1$ many $x_{l(\lambda)}$ appear.
	Therefore, $h_T\cdot(p-1\ p\ j)\neq h_T$,
	concluding the proof.
\end{proof}
With all this preparation, we are now ready to proof our main result,
which in a slightly modified form was conjectured by Kahle and Micha\l ek in \cite[Conj. 4.3]{KM16} for arbitrary $p$ and all $\lambda$,
and proposed to Kahle and Micha\l ek by Mich\`ele Vergne (private communication with the second author of \cite{KM16}).

In \cite[Lemma 4.1]{KM16} a proof for \glqq non exceptional\grqq $\lambda$  whose parts are all distinct is given.
\begin{theorem}\label{connection:MainThm}
	Let $p,k\in\N$, and $\lambda\vdash pk$ with $l(\lambda)\leq p$.
	Then,
	\begin{enumerate}[(i)]
	\item
		if $\lambda$ is of the form (\glqq exceptional\grqq)
		\begin{displaymath}
			(pk),(k^{p}),(a^{p-1}),(b,c^{p-1}),(b^{p-1},c),
		\end{displaymath}
		we either have
		\begin{displaymath}
			a_{(p),(2dk)}^{2d\lambda}=a_{(1^{p}),((2d+1)k)}^{(2d+1)\lambda}=1,\quad a_{(p),((2d+1)dk)}^{(2d+1)\lambda}=a_{(1^p),(2dk)}^{2d\lambda}=0,\quad
			a_{\mu,(dk)}^{d\lambda}=0
		\end{displaymath}
		for all $d\geq 0$ and $\mu\vdash p$, $\mu\neq (p),(1^p)$, or
		\begin{displaymath}
			a_{(p),(dk)}^{d\lambda}=1,\quad
			a_{\mu,(dk)}^{d\lambda}=0
		\end{displaymath}
		for all $d\geq 0$ and $\mu\vdash p$, $\mu\neq(p)$,
	\item
		if $d=4$ and $\lambda=(2k,2k)$, then
		\begin{displaymath}
		\begin{array}{ll}
			a_{(4),(d)}^{(2d^2)}=
			\left\lfloor\frac{2d}{3}\right\rfloor-\frac{d}{2}+
			\begin{cases}
			1&d\mathrm{\ even}\\
			\frac{1}{2}&d\mathrm{\ odd}
			\end{cases},\quad &
			a_{(1^4),(d)}^{(2d^2)}=\left\lfloor\frac{2d}{3}\right\rfloor-\frac{d}{2}+
			\begin{cases}
			0&d\mathrm{\ even}\\
			\frac{1}{2}&d\mathrm{\ odd}
			\end{cases},\\[\bigskipamount]
			a_{(2,2),(d)}^{(2d^2)}=d-\left\lfloor\frac{2d}{3}\right\rfloor,&
			a_{(3,1),(d)}^{(2d^2)}=a_{(2,1^2),(dk)}^{(2d^2)}=0,
		\end{array}
		\end{displaymath}
		and if $\lambda=(b^2,c^2)$ for $b>c$, then $a_{\mu,(dk)}^{d\lambda}=a_{\mu,(d(k-a)}^{((b-c)^2))}$,
	\item
		and else $a_{\mu,(dk)}^{d\lambda}$ is a quasi-polynomial in $d$ of the same (positive) degree as $c_{p,dk}^{d\lambda}$ with constant leading term equal to $\frac{\dim(V_\mu)}{p!}$ times the leading term of $c_{p,dk}^{d\lambda}$
		for every $\mu\vdash p$.
	\end{enumerate}
\end{theorem}
\begin{proof}
	Let $\bigoplus_{d\geq 0}B_d$ be the graded algebra and $\beta_d:S_p\to\GL(B_d)$ the representations from \cref{connection:prop1},
	and let $PK$ be the subgroup of $S_p$ defined in \cref{connection:Howe}.
	We consider the case $p\neq 4$.

	First, suppose $\lambda$ is not of the form (\glqq exceptional\grqq)
	\begin{displaymath}
		(pk),(k^{p}),(a^{p-1}),(b,c^{p-1}),(b^{p-1},c),
	\end{displaymath}
	and let $\sigma\in PK$, $d\geq0$, $\tau\in S_p$,
	and $c_\sigma\in B_d$ with $\beta_d(\sigma)=c_\sigma\id$.
	
	Moreover, assume that $l(\lambda)\leq p-1$ and that $\lambda$ is not of the form $(pk),(a^{p-1})$.
	If we had $A_p\subset PK$,
	then for every $d\geq0$ only the sign representation and the trivial representation of $S_p$ would appear in $B_d$ by \cref{connection:cor3},
	on whom $A_p$ acts trivially.
	But $B_d$ is the space of highest weight vectors of weight $d\lambda$ in $(S^{dk}(V))^{\otimes p}$ by \cref{connection:prop1},
	and by \cref{connection:prop3} there is an even permutation $\sigma\in S_p$ and a highest weight vector $h$ of weight $\lambda$ such that $h\cdot\sigma\neq h$.	
	Therefore, we have $PK=\{1\}$,
	and \cref{connection:Howe} together with \cref{connection:prop1} implies
	\begin{displaymath}
		\lim\limits_{d\to\infty}\frac{a_{\mu,(dk)}^{d\lambda}}{c_{p,dk}^{d\lambda}}=\frac{\dim(V_\mu)}{p!}
	\end{displaymath}
	for any $\mu\vdash p$.
	
	Furthermore, $a_{\mu,(dk)}^{d\lambda}$ is a quasi-polynomial by \cref{plethysm:QuasiPolynomial},
	and $c_{p,dk}^{d\lambda}$ is a quasi-polynomial of positive degree with constant leading term by \cref{connection:cor1}.
	This implies that $a_{\mu,(dk)}^{d\lambda}$ is a quasi-polynomial in $d$ of the same (positive) degree as $c_{p,dk}^\lambda$ with constant leading term equal to $\frac{\dim(V_\mu)}{p!}$.
	
	Now assume $l(\lambda)=p$ and that $\lambda$ is not of the form $(k^p),(b^{p-1},c),(b,c^{p-1})$.
	Then
	\begin{displaymath}
		\lambda':=(\lambda_1-\lambda_p,\ldots,\lambda_{p-1}-\lambda_p)\vdash p(k-\lambda_p),
	\end{displaymath}
	is a partition with $l(\lambda')\leq p-1$ not of the form $(pk),(a^{p-1})$.
	Furthermore,
	\begin{displaymath}
		a_{\mu,(dk)}^{d\lambda}\overset{\cref{plethysm:cor1}}{=}
		\begin{cases}
			a_{\mu,(d(k-\lambda_p))}^{d\lambda'}&,\lambda_p\mathrm{\ even}\\
			a_{\mu^T,(d(k-\lambda_p))}^{d\lambda'}&,\lambda_p\mathrm{\ odd}
		\end{cases},\quad
		c_{p,dk}^{d\lambda}\overset{\cref{tensorasymp:cor1}}{=}
		c_{p,d(k-\lambda_p)}^{d\lambda'}
	\end{displaymath}
	for any $\mu\vdash p$.
	As $\dim(V_\mu)=\dim(V_{\mu^T})$ by the hook length formula \cite[4.12]{FH04},
	this yields the claim for $\lambda$ with $l(\lambda)=p$.
	
	Now assume $\lambda$ is of the form
	\begin{displaymath}
			(pk),(1^{pk}),(a^{p-1}),(b,c^{p-1}),(b^{p-1},c).
	\end{displaymath}
	For any $d\geq 0$, by \cref{tensorasymp:prop1} and \cref{connection:prop1}
	\begin{equation}\label{connection:eq2}
		\sum\limits_{\mu\vdash p}a_{\mu,(dk)}^{d\lambda}\dim(V_\mu)=
		\dim(B_{d})=c_{p,dk}^{d\lambda}=1.
	\end{equation}
	If we had $\{1\}=PK$,
	then by \cref{connection:Howe}
	\begin{displaymath}
		\lim\limits_{d\to\infty}\frac{a_{\mu,(dk)}^{d\lambda}}{c_{p,dk}^{d\lambda}}=\frac{\dim(V_\mu)}{p!}
	\end{displaymath}
	for any $\mu\vdash p$, yielding a contradiction to \cref{connection:eq2}, as there are multiple partitions of $p\neq 1$.
	Therefore, we have $PK\neq\{1\}$ and for any $d\geq 0$ only the sign or trivial representation of $S_p$ appears in $B_d$ by \cref{connection:cor3},
	and \cref{connection:eq2} yields that $B_d$ is either the sign or trivial representation for any $d\geq 0$.
	
	Furthermore,
	for any $d\geq0$ we have $a_{(p),(2dk)}^{2d\lambda}\geq 1$ by Weintraub's conjecture \cref{plethysm:Weintraub},
	which together with \cref{connection:eq2} yields $a_{(p),(2dk)}^{2d\lambda}=1$ and $a_{\mu,(2dk)}^{2d\lambda}=0$ for any $\mu\vdash p$, $\mu\neq (p)$.
	
	Now assume that $B_1$ is the sign representation of $S_p$.
	Then $B_1\cdot B_{2d}=B_{2d+1}$ is the sign representation for any $d\geq 0$, as $B_{2d}$ is the trivial representation, i.e., $a_{(1^p),((2d+1)k)}^{(2d+1)\lambda}=1$,
	and \cref{connection:eq2} implies $a_{\mu,((2d+1)k)}^{(2d+1)\lambda}=0$ for any $d\geq 0$ and $\mu\vdash p$, $\mu\neq (1^p)$.
	
	On the other hand, if $B_1$ is the trivial representation, than $B_1\cdot B_{2d}=B_{2d+1}$ is the trivial representation for any $d\geq 0$, as $B_{2d}$ is the trivial representation, i.e., $a_{(p),((2d+1)k)}^{(2d+1)\lambda}=1$,
	and \cref{connection:eq2} implies $a_{\mu,((2d+1)k)}^{(2d+1)\lambda}=0$ for any $d\geq 0$ and $\mu\vdash p$, $\mu\neq (p)$.
	
	Lastly, for $p=4$ by replacing $A_4$ by $V$ the argument works mutatis mutandis for all partitions apart from $(2k^2),(b^2,c^2)$, $b>c$.
	The case $b>c$ however reduces as above to $\lambda=(2,2)$, and then one can derive explicit formulas using the computations of Kahle-Micha\l ek, see \
	\begin{figure}[H]
	\centering
	\href{https://www.thomas-kahle.de/plethysm.html}{https://www.thomas-kahle.de/plethysm.html}
	\end{figure}
	as well as the appendix of the arXiv-version of their paper \cite{KM15}.	
	This concludes the proof.
\end{proof}
\begin{remark}
	When constructing highest weight vectors for $p=4$ and multiples of $\lambda=(2,2)$,
	one only gets highest weight vectors on which $V\subset S_4$ acts trivially.
	Since the Specht modules on which $V$ acts trivially are exactly those for $\mu=(4),(1^4),(2,2)$, with $V_{(2,2)}$ given by inflating the standard representation of $S_3$ along $S_4\to S_4/V\cong S_3$, this matches the formulas we get.
\end{remark}

\bibliographystyle{amsplain}
\bibliography{Asymptotics_of_Plethysm_bib}

\end{document}